\documentclass[11pt,a4paper]{amsart}
\usepackage[english]{babel}
\usepackage[utf8]{inputenc}

\usepackage{microtype}
\usepackage{mathtools}
\usepackage{amssymb,amsmath}
\usepackage{graphicx}
\usepackage{bm,bbm}
\usepackage{color}
\usepackage{tikz}
\usepackage{ifthen}
\usetikzlibrary{snakes}
\usetikzlibrary{patterns}
\usepackage{algorithm}
\usepackage{setspace}
\usepackage[noend]{algpseudocode}
\usepackage{pgf}
\usepackage{pgfplots}
\usepackage{enumitem}
\usepackage{array,blkarray}
\usepackage{bigdelim}
\usepackage{listings} 
\usepackage[
]
{hyperref}
\usepackage{amssymb}
\usepackage{isotope}
\usepackage{stmaryrd}
\usepackage[]{placeins}
\usepackage[normalem]{ulem}
\usepackage{caption}
\usepackage{subcaption}
\newcolumntype{L}[1]{>{\raggedright\let\newline\\\arraybackslash\hspace{0pt}}m{#1}}
\newcolumntype{C}[1]{>{\centering\let\newline\\\arraybackslash\hspace{0pt}}m{#1}}
\newcolumntype{R}[1]{>{\raggedleft\let\newline\\\arraybackslash\hspace{0pt}}m{#1}}
\newtheorem{theorem}{Theorem}
\newtheorem{proposition}[theorem]{Proposition}

\theoremstyle{definition}

\newtheorem{lemma}[theorem]{Lemma}

\theoremstyle{remark}

\numberwithin{equation}{section}
\numberwithin{table}{section}
\numberwithin{figure}{section}
\numberwithin{algorithm}{section}
\definecolor{myBlue}{RGB}{113,104,238} 
\definecolor{myGreen}{RGB}{154,205,50} 
\definecolor{myGreen2}{RGB}{114,175,30} 
\definecolor{myRed}{RGB}{180,50,50}  
\definecolor{myOrange}{RGB}{225,92,22} 
\definecolor{lgray}{RGB}{200,200,200} 
\definecolor{llgray}{RGB}{155,155,155} 
\definecolor{mycolor1}{rgb}{0.00000,0.44700,0.74100}%
\definecolor{mycolor2}{rgb}{0.85000,0.32500,0.09800}%
\definecolor{mycolor3}{rgb}{0.92900,0.69400,0.12500}%
\definecolor{mycolor4}{rgb}{0.49400,0.18400,0.55600}%
\definecolor{mycolor5}{rgb}{0.46600,0.67400,0.18800}%
\definecolor{mycolor6}{rgb}{0.30100,0.74500,0.93300}%
\definecolor{mycolor7}{rgb}{0.63500,0.07800,0.18400}%
\newcommand\R{\mathbb R}
\newcommand\Z{\mathbb Z}

\DeclareMathOperator{\supp}{supp}

\DeclareMathOperator{\diag}{diag}

\DeclareMathOperator{\tr}{tr}

\newlength\figurewidth
\usepackage{mathtools}
\DeclareMathAlphabet{\dutchcal}{U}{dutchcal}{m}{n}
\SetMathAlphabet{\dutchcal}{bold}{U}{dutchcal}{b}{n}
\DeclareMathAlphabet{\dutchbcal} {U}{dutchcal}{b}{n}
\newcommand{\M}{\mathcal{M}}
\newcommand{\m}{\dutchcal{m}}
\newcommand{\adj}{\operatorname{adj}}
\newcommand{\pc}{{\operatorname{pc}}}

\newcommand{\qc}{{\operatorname{qc}}}
\renewcommand{\c}{{\operatorname{c}}}

\newcommand{\lscc}{\operatorname{lsc-c}}
\newcommand{\lscpc}{\operatorname{lsc-pc}}
\newcommand{\lscsvpc}{\operatorname{lsc-svpc}}

\newcommand{\SO}{\mathcal{SO}}
\newcommand{\Pid}{\Pi_d}
\DeclareMathOperator{\dev}{dev}
\begin{document}
	\title[Computational polyconvexification of isotropic functions]{Computational polyconvexification of isotropic functions}
	\author[]{T.~Neumeier$^{*}$, M.~A.~Peter$^{\dagger}$, D.~Peterseim$^{\dagger}$, D.~Wiedemann$^{*}$}
	\address{${}^{*}$ Institute of Mathematics, University of Augsburg, Universit\"atsstr.~12a, 86159 Augsburg, Germany}
	\email{\{timo.neumeier, david.wiedemann\}@uni-a.de}
	\address{${}^{\dagger}$ Institute of Mathematics \& Centre for Advanced Analytics and Predictive Sciences (CAAPS), University of Augsburg, Universit\"atsstr.~12a, 86159 Augsburg, Germany}
	\email{\{malte.peter, daniel.peterseim\}@uni-a.de}
	
	\thanks{T.~Neumeier, M.~A.~Peter and D.~Peterseim gratefully acknowledge funding from the German Research Foundation (DFG) within the Priority Programme 2256 \emph{Variational Methods for Predicting Complex Phenomena in Engineering Structures and Materials} (Reference IDs PE1464/7-1 and PE2143/5-1). D.~Wiedemann thanks the Free State of Bavaria for funding through the Marianne--Plehn--Program.
	}
	\date{\today}

	\begin{abstract}
		Based on the characterization of the polyconvex envelope of isotropic functions by their signed singular value representations, we propose a simple algorithm for the numerical approximation of the polyconvex envelope. Instead of operating on the $d^2$-dimensional space of matrices, the algorithm requires only the computation of the convex envelope of a function on a $d$-dimensional manifold, which is easily realized by standard algorithms. The significant speedup associated with the dimensional reduction from $d^2$ to $d$ is demonstrated in a series of numerical experiments. 
	\end{abstract}

	\maketitle


\section{Introduction}\label{sec:intro}
Many applications in the field of nonlinear elasticity aim at finding a global minimizer of functionals of the form
\begin{equation} \label{eq:energyFunctional}
I(u) = \int_{\Omega} W\left(\nabla u(x)\right) \, \mathrm{d} x
\end{equation}
over a domain $\Omega \subset \R^{d}$ in spatial dimension $d \in \{2, 3\}$ for a suitable weak class of deformations ${u\colon\Omega\rightarrow \R^d}$. 
In many relevant cases, the density ${W\colon\R^{d \times d} \to \R_\infty \coloneqq \R \cup \{\infty\}}$ does not satisfy a suitable notion of convexity and the existence of minimizers cannot be guaranteed. In fact, the infimum may not be reached and nonconvexity, e.g., a multiwell structure of the energy density, may lead to the emergence of increasingly fine microstructures within the minimizing sequences. 
Moreover, the application of standard discretization methods to describe oscillations in the infimizing sequences of $I$ typically leads to mesh-dependent results with oscillations in the discrete deformation gradient at the length scale of the mesh size. Therefore, alternative approaches are introduced for both mathematical analysis and numerical simulation using relaxed formulations that focus on macroscopic features responsible for global behavior by extracting the relevant information from the unresolved microstructures \cite{BarCarHacHop04, ConDol18, KNMPPB2022, KNPPB2023, BKNPP23}.

The direct method in the calculus of variations links the limit behavior of a minimizing sequence of $I$ to the minimizer of the function when $W$ is replaced by its quasiconvex envelope $W^\textrm{qc}$ \cite{Mue99,Dac08,Rou20}. 
Since the quasiconvex envelope $W^\textrm{qc}$ is rarely known explicitly or even approximately, lower and upper bounds of $W^\textrm{qc}$ and their numerical approximation are of interest in computational nonlinear elasticity. 
While the rank-one convex envelope provides an upper bound, a lower bound of $W^\qc$ is provided by the polyconvex envelope $W^\pc$. 
The notion of quasiconvexity requires growth conditions for the existence of a minimizer, which reflect a rather unphysical behavior in the compression regime ($\det(F) \to 0$). 
Such growth conditions are not necessary under polyconvexity \cite{Bal76, Bal77} and, therefore, polyconvexity is favorable for application in nonlinear elasticity \cite{Bal02}.
Overall, the accurate approximation of polyconvex envelopes is of great importance for the practical realization of relaxation techniques.

In the context of elasticity, the main challenge for the efficient computation of semiconvex envelopes like polyconvex ones arises from the high-dimensional nature of the problem. Even the accurate numerical representation of the original density $W$ requires a mesh of a domain in $d\times d$-dimensional space. For $d=2$, this is already challenging but just about feasible with known linear programming algorithms for approximating the polyconvex envelope \cite{Bar05}, \cite{EBG13}, \cite{BEG15}. However, even these efficient algorithms become practically infeasible for many relevant problems in $d=3$ spatial dimensions.

In this work, we will therefore abandon the generality of these methods in favor of faster algorithms by restricting ourselves to the subclass of isotropic functions $W$. These functions model a directionally independent local material response.
They can be identified with a function $\Phi\colon\R^d \rightarrow \R_{\infty}$ by means of the vector of signed singular values ${\nu(F) \in \R^{d}}$ of $F$, namely $W(F)=\Phi(\nu(F))$.
In order to utilize this dimension reduction, it becomes desirable to characterize also polyconvexity in terms of $\Phi$. Due to the high relevance of isotropic functions in nonlinear elasticity, this task was already addressed during the introduction of polyconvexity in \cite{Bal77} and a sufficient condition for the polyconvexity in terms of the singular values was derived there. However, it is a priori not necessary and, thus, not suited for the computation of the lower polyconvex envelope. Sufficient and necessary conditions are presented, for $d=2$, in \cite{Sil97, Ros98, Sil99} and, for $d=\{2,3\}$, in \cite{Mie05}. However, these conditions are not directly accessible for the numerical polyconvexification due to their implicit structure. 
For $d=2$, a characterization of finite isotropic polyconvex functions by means of merely a convex symmetric function was presented in \cite{DM06}. Finally, for $d\in \{2,3\}$, such a characterization was achieved in \cite{WiePet23}, which can handle also functions attaining infinity. This characterization corresponds to the definition of polyconvexity restricted to the set of diagonal matrices and, thus, can be considered optimal for isotropic functions.
For functions arising in linear elasticity, results on dimension reduction are derived in \cite{BKS19}.

The characterization of polyconvexity of isotropic functions presented in \cite{WiePet23} leads to a simple algorithm for approximating their polyconvex envelope. Given a mesh of a $d$-dimensional domain of signed singular values, the algorithm simply lifts the mesh to a $d$-dimensional manifold embedded in $3$- and $7$-dimensional space for $d=2$ and $d=3$, respectively. The subsequent computation of the convex hull can be easily done with various algorithms such as \texttt{Quickhull} \cite{BarDobHuh96} or the ones based on linear programming mentioned above. The computational effort is determined by the number of mesh points, which scales as in $d$-dimensional space. This dimensional reduction from $d\times d$ to $d$ dimensions makes the novel algorithms feasible for the approximation of the polyconvex envelopes even for engineering applications in three dimensions.

The paper is organized as follows.
In Section~\ref{sec:theory}, basic definitions, the analytical theory of polyconvexity as well as necessary and sufficient conditions based on \cite{WiePet23} are presented. 
Afterwards, in Section \ref{sec:svpoly}, the numerical realization of the polyconvexification approach, i.e., discretization and algorithmic treatment, are discussed.
In Section~\ref{sec:examples}, a collection of numerical experiments shows the feasibility of the algorithms even in three spatial dimensions. 
Moreover, we numerically investigate polyconvexity properties of a parameter-dependent family of exponentiated Hencky-logarithmic energy densities \cite{NefLanGhiMarSte2015} beyond existing mathematical results. 
We conclude with some remarks in Section~\ref{sec:conclusion}.

\section{Polyconvexification of isotropic functions}\label{sec:theory}
Let $d\in \{2,3\}$ and let $W \colon \R^{d \times d} \to \R_\infty \coloneqq \R \cup \{\infty\}$ be a function, which maps $d\times d$ matrices to  real scalars or infinity. We think of energy densities in functionals of the form \eqref{eq:energyFunctional}. The possible value infinity models practically  unrealizable states of the deformation gradient. 
We are interested in the polyconvex envelope $W^{\pc}$ of the function $W$. The notion of polyconvexity relies on the minors of matrices $F\in \R^{d\times d}$. Given the determinant $\det(F)$ and the adjoint $\adj(F)$ of $F$, let 
\begin{equation}\label{eq:M(F)}
	\M(F)\coloneqq \begin{cases} (F, \det(F))& \text{if } d=2,\\
	 (F, \adj(F), \det(F))&  \text{if } d=3, 
	 \end{cases}
 \end{equation}
denote the minors of $F$. Since, for $d = 2$, we identify $\R^{2 \times 2} \times \R \cong \R^5$ and, for ${d = 3}$, we identify $\R^{3 \times 3} \times \R^{3 \times 3} \times \R \cong \R^{19}$, $\M(F)$ is considered as a vector of dimension ${K_d \coloneqq \sum_{l =1 }^d \binom{d}{l}^{2}}$.
A function ${V\colon\R^{d \times d} \to \R_{\infty}}$ is said to be polyconvex if there exists a convex function ${G\colon\R^{K_d} \to \R_{\infty}}$ such that for all $F\in \R^{d\times d}$,  
\begin{align} \label{eq:V=GM}
	V(F) = G(\M(F)).
\end{align}
The polyconvex envelope $W^{\pc}\colon \R^{d \times d} \to \R_\infty$ of $W$, defined by the pointwise supremum
\begin{align} \label{eq:Wpc}
	W^{\pc}(F) \coloneqq \sup \left\{V(F) \mid V \colon \R^{K_d} \to \R_{\infty} \text{ polyconvex},  V \leq W\right\},
\end{align}
is the largest polyconvex function below $W$. It 
is equivalently characterized by
\begin{align*}
	W^{\pc}(F) = \sup \left\{(G \circ \M) (F) \mid G: \R^{K_d} \to \R_{\infty} \text{ convex},  G \circ \M \leq W\right\}.
\end{align*}
Given the (non-convex) function $H \colon \R^{K_d} \to \R$ with
\begin{align} \label{eq:H}
	H(x)= 
	\begin{cases}
		W(F) & \text{ if } x = \M(F), \\
		\infty & \text{ else},
	\end{cases}
\end{align}
the polyconvex envelope	$W^{\pc}$ of $W$ equals the convex envelope $H^{\c}$ of $H$, i.e.,
\begin{align} \label{eq:Wpc=HcM}
	W^{\pc} (F) = H^{\c}(\M(F)).
\end{align}

Note that for finite valued functions convexity implies continuity.  
However, this is not the case for functions taking the value $\infty$.
Indeed, for the study of the functional \eqref{eq:energyFunctional}, it is important that the function $G$ in the definition of polyconvexity \eqref{eq:V=GM} is lower semicontinuous (lsc).
In this sense, we call a function $V$ lower semicontinuous polyconvex (lsc-pc) if $G$ in \eqref{eq:V=GM} is additionally lower semicontinuous. 
Accordingly, we define the lower semicontinuous polyconvex envelope $W^{\lscpc}$ of $W$ to be the largest lower semicontinuous polyconvex function below $W$. 

In this paper, we restrict ourselves to the study of polyconvexity of isotropic functions. According to \cite{Bal76}, ${W\colon\R^{d\times d}\rightarrow \R_{\infty}}$ is called \textit{objective} if $W(F) = W(R F)$ for all $F \in \R^{d \times d}$ and for all $R \in \mathcal{SO}(d)$, where $\mathcal{SO}(d)$ denotes the special orthogonal group of $d \times d$ matrices. Furthermore, $W$ is called \textit{isotropic} if $W$ is objective and $W(F) = W(Q F Q^T)$ holds for all $F$ and for all $Q\in \mathcal{O}(d)$, the group of orthogonal $d\times d$ matrices.
Therefore, $W$ is isotropic if and only if 
\begin{align*}
W(F) = W(R_1 F R_2)
\end{align*}
for all $F\in \R^{d \times d}$ and all $R_1, R_2 \in \SO(d)$.

Isotropic functions can be characterized by the signed singular values of their arguments. Given $F\in\R^{d\times d}$ with singular values $0\leq \sigma_1(F),\ldots,\sigma_d(F)\in\R^d$, the \textit{signed singular values} $\nu_1(F)=\varepsilon_1 \sigma_1(F), \ldots, \nu_d(F)=\varepsilon_d\sigma_d(F) \in \R$ of $F$ have the same absolute values as the singular values of $F$ and the signs $\varepsilon_1,\ldots,\varepsilon_d\in\{1,0,-1\}$ satisfy $$\operatorname{sign}(\nu_1\cdot\ldots\cdot\nu_d) = \varepsilon_1\cdot\ldots\cdot\varepsilon_d = \operatorname{sign}(\det(F)).$$
Note that the signed singular values are only unique up to permutations in 
\begin{align*}
	\Pid = \left\{P \diag(\varepsilon) \in \mathcal{O}(d) \mid P \in \operatorname{Perm}(d), \varepsilon \in \{-1,1\}^d, \varepsilon_1\cdot\ldots\cdot\varepsilon_d = 1 \right\},
\end{align*}
where $\diag(\bullet)$ refers to the diagonal matrix with diagonal entries given by the vector of its argument and $\operatorname{Perm}(d) \subset \{0,1\}^{d \times d}$ denotes the set of  permutation matrices.
By means of the signed singular values, we can identify the set of isotropic functions $W\colon \R^{d \times d} \to \R$ with the set of $\Pid$-invariant functions $\Phi\colon\R^d \to \R_\infty$, i.e., ${\Phi(\hat{\nu}) = \Phi(S \hat{\nu})}$ for all $\hat{\nu}\in \R^{d}$ and all $S \in \Pid$.
The identification is given by
\begin{equation}\label{eq:W=Phi}
	W(F) = \Phi(\nu(F))
\end{equation}
for all $F\in \R^{d\times d}$ and vice versa
\begin{equation}\label{eq:Phi=W}
	\Phi(\hat{\nu}) = W(\diag(\hat{\nu}))
\end{equation}
for all $\hat{\nu} \in \R^{d}$. 
Given this identification, we say that a $\Pid$-invariant function ${\Phi\colon\R^d \to \R_\infty}$ is singular value polyconvex if the corresponding $W$ defined by \eqref{eq:W=Phi} is polyconvex.
Accordingly, a $\Pi_d$-invariant function $\Phi$ is called lower semicontinuous singular value polyconvex (lsc-svpc) if $W$ is lower semicontinuous polyconvex. We define the lower semicontinuous singular value polyconvex envelope $\Phi^{\lscsvpc}\colon\R^{d} \to \R_\infty$ of $\Phi$ by
\begin{align*}
	\Phi^{\lscsvpc} (\hat{\nu}) \coloneqq \sup \{\Psi(\hat{\nu}) \mid \Psi \text{ lsc-svpc}, \Psi \leq \Phi\}.
\end{align*}
This definition is justified by the following lemma, which shows that the lower semicontinuous polyconvex envelope preserves isotropy and, therefore, naturally connects the polyconvex envelopes of $W$ and $\Phi$.
\begin{lemma} \label{thm:Wpcisotropic}
	Let $W\colon\R^{d \times d} \to \R_\infty$ be isotropic and let $\Phi\colon\R^d \to \R_\infty$ be the unique {$\Pid$-invariant} function that satisfies \eqref{eq:W=Phi} or, equivalently, \eqref{eq:Phi=W}. 
	Then, $W^{\lscpc}$ is isotropic and identified with $\Phi^{\lscsvpc}$, i.e.,
	\begin{align*} 
		W^{\lscpc}(F) = \Phi^{\lscsvpc}(\nu(F))\quad\text{and}\quad \Phi^{\lscsvpc}(\hat{\nu}) = W^{\lscpc}(\diag(\hat{\nu}))
	\end{align*}
	hold for all $F \in \R^{d \times d}$ and for all $\hat{\nu} \in \R^d$.
\end{lemma}
\begin{proof}
	For every lower semicontinuous polyconvex, not necessarily isotropic $V$ with $V \leq W$, there exists an isotropic lower semicontinuous polyconvex function $V_{\textrm{iso}}$ with $V \leq V_{\textrm{iso}}\leq W$, namely, $V_{\textrm{iso}}(F) \coloneqq \sup_{R_1, R_2 \in \mathcal{SO}(d)} V(R_1 F R_2)$ for all $F \in \R^{d \times d}$.
	Using $V \leq W$ and the isotropy of $W$, we observe
	\begin{align*}
		V_{\textrm{iso}}(F)= \sup_{R_1, R_2 \in \mathcal{SO}(d)} V(R_1 F R_2) \leq \sup_{R_1, R_2 \in \mathcal{SO}(d)} W(R_1 F R_2) = W(F)
	\end{align*} 
	for all $F \in \R^{d \times d}$, i.e.~$V_{\textrm{iso}} \leq W$. By its definition, $V_{\textrm{iso}}$ is isotropic and lower semicontinuous polyconvex since convexity and lower semicontinuity are preserved for the supremum.
	
	Consequently, it suffices to consider the supremum in \eqref{eq:Wpc} over isotropic lower semicontinuous polyconvex functions $V_{\textrm{iso}}$, which can be identified by $\Pid$-invariant functions, i.e.,
	\begin{align*}
		W^{\lscpc}(F) & = \sup\{V(F) \mid V \text{ lsc-pc}, V \leq W\} \\
		& = \sup\{V_{\textrm{iso}}(F) \mid V_{\textrm{iso}} \text{ lsc-pc and isotropic}, V_{\textrm{iso}} \leq W\} \\
		& = \sup\{\Psi(\nu(F)) \mid \Psi \text{ lsc-svpc}, \Psi \circ \nu \leq \Phi \circ \nu\} \\
		& = \sup\{\Psi(\nu(F)) \mid \Psi \text{ lsc-svpc}, \Psi\leq \Phi\} \\
		& = \Phi^{\lscsvpc} (\nu(F)),
	\end{align*}
	where we employ the surjectivity of $\nu : \R^{d\times d} \to \R^d$.
	In particular, this shows that $W^{\lscpc}$ is isotropic.
\end{proof}

In analogy to the original definition of polyconvexity based on the minors $\M$ in \eqref{eq:M(F)}, we define a lifting of the arguments to a higher dimensional space. For $d \in \{2,3\}$, we define $k_d \coloneqq 2^d-1$ and the mapping ${\m \colon\R^{d} \to \R^{k_d}}$ by
\begin{align*}
	\m(\hat{\nu}) = 
	\begin{cases}
		(\hat{\nu}_1, \hat{\nu}_2, \hat{\nu}_1 \hat{\nu}_2) & \text{if  } d = 2, \\
		(\hat{\nu}_1, \hat{\nu}_2, \hat{\nu}_3, \hat{\nu}_2 \hat{\nu}_3, \hat{\nu}_3 \hat{\nu}_1, \hat{\nu}_1 \hat{\nu}_2, \hat{\nu}_1 \hat{\nu}_2 \hat{\nu}_3) & \text{if  } d = 3.
	\end{cases}
\end{align*}
We will refer to $\m(\hat{\nu})$ as the vector of minors of $\hat{\nu}\in\R^d$. 
According to \cite{WiePet23}, the lower semicontinuous singular polyconvexity of a $\Pid$-invariant function $\Phi\colon\R^d \to \R_\infty$ can be characterized by the existence of a convex function acting on the ambient space $\R^{k_d}$ of the image $\mathfrak{m}_d = \{\m(\hat{\nu}) \mid \hat{\nu} \in \R^d \}$ of the lifting $\m$.

\begin{theorem}[{\cite[Theorem 1.1]{WiePet23}}] \label{thm:SingConj}
	Let $d \in \{2, 3\}$ and $\Phi\colon\R^d \to \R_{\infty}$ be $\Pid$-invariant.
	Then, $\Phi$ is lower semicontinuous singular value polyconvex if and only if there exists a lower semicontinuous and convex function $g: \R^{k_d} \to \R_\infty$ such that
	\begin{align}\label{eq:Phi=gm}
		\Phi = g \circ \m.
	\end{align}
\end{theorem}
Note that, on one hand, the assumption on $\Phi$ to be $\Pid$-invariant is already ensured if $\Phi$ is lower semicontinuous singular value polyconvex. 
On the other hand, it can be ensured if \eqref{eq:Phi=gm} holds and $g(\m(\hat{\nu})) = g(\m(S \hat{\nu}))$ for all $S \in \Pid$ and $\hat{\nu} \in \R^d$.

Similar to \eqref{eq:Wpc=HcM}, we can construct the lower semicontinuous singular value polyconvex envelope by lifting the signed singular values onto the manifold $\mathfrak{m}_d$ of their minors.
\begin{proposition} \label{thm:main}
	Let $\Phi \colon \R^d \to \R_\infty$ be $\Pid$-invariant. Define the mapping
	\begin{align} \label{def:h}
		h \colon \R^{k_d} \to \R_\infty,\quad x \mapsto
		\begin{cases}
			\Phi(\hat{\nu}) &\textrm{if } x = \m(\hat{\nu}), \\
			\infty &\textrm{else}
		\end{cases}
	\end{align}
and denote by $h^{\lscc}$ its lower semicontinuous convex envelope. 
Then 
\begin{align} \label{eq:Phipc=hc}
	\Phi^{\lscsvpc} = h^{\lscc}\circ\m.
\end{align}
\end{proposition}

\begin{proof}
By means of Theorem \ref{thm:SingConj}, we obtain 
\begin{equation*} \label{eq:lsc-svpc-envelope}
	\begin{aligned}
		\Phi^{\lscsvpc}(\hat{\nu})&=\sup\{\Psi(\hat{\nu}) \mid \Psi \text{ lsc-svpc}, \Psi\leq\Phi\} \\
		&=\sup\left\{g_{\textrm{sym}}(\m(\hat{\nu}))\; \bigg\vert \begin{array}{l}
			g_{\textrm{sym}}\text{ lsc-c},\;g_{\textrm{sym}}\leq h,\\
			\forall S \in \Pid:   g_{\textrm{sym}} \circ \m =  g_{\textrm{sym}} \circ \m \circ S
		\end{array}\right\}.
	\end{aligned}
\end{equation*}
Actually, the symmetry assumption on the right-hand side is redundant. For any lower semicontinuous and convex function $g\leq h$, define the function $g_{\textrm{sym}}$ with ${g_{\textrm{sym}}\circ \m = g_{\textrm{sym}}\circ\m\circ S}$ for all $S \in \Pid$ by 
\begin{align*}
	g_{\textrm{sym}}(x) \coloneqq \max_{S \in \Pi_d} \begin{cases} g(S(x_1, x_2)^\top, x_3)& \text{if }d=2,\\g(S(x_1, x_2, x_3)^\top, S(x_4, x_5, x_6)^\top, x_7)& \text{if }d=3.\end{cases}
\end{align*}
By construction, $g_{\textrm{sym}}$ is lower semicontinuous and convex. Moreover, the $\Pi_d$-invariance of $h$ shows that for any $\hat{\nu}\in\R^d$, 
$$g_{\textrm{sym}}(\hat{\nu}) = g(S_{\hat{\nu}} \hat{\nu}) \leq h(S_{\hat{\nu}}\hat{\nu})= h(\hat{\nu})$$
holds with $S_{\hat{\nu}} = \operatorname{argmax}_{S\in \Pi_d} g(S \hat{\nu})$. It follows that $g \leq g_{\textrm{sym}}\leq h$ and, altogether, the claimed assertion holds, 
\begin{align*}
	\Phi^{\lscsvpc}(\hat{\nu}) &= \sup\{g(\m(\hat{\nu})) \mid g\text{ lsc-c},\;g\leq h\} = h^{\lscc} (\m(\hat{\nu})).\qedhere
\end{align*}
\end{proof}

\section{Computational signed singular value polyconvexification}\label{sec:svpoly}
The polyconvexity of a general function $W$ acting on $d\times d$-matrices via the definition \eqref{eq:Wpc=HcM} requires, in the absence of structural properties such as isotropy, the computation of the convex envelope of a scalar function $H$ acting essentially on a $d^2$-dimensional manifold in $5$- and $19$-dimensional space for $d=2$ and $d=3$, respectively. So the suitable representation of $W$ on any computational mesh already suffers severely from the high dimension, the actual computational convexification even more, no matter which algorithm is used. 

This often prohibitively high computational cost can be reduced considerably under the structural assumption of isotropy. In this case, Lemma~\ref{thm:Wpcisotropic} shows that it suffices to compute the lower semicontinuous singular value polyconvex envelope of a $\Pid$-invariant function $\Phi$ acting on the signed singular values.
By Proposition~\ref{thm:main}, this can be done by computing the convex envelope $h^{\c}$ of the function $h$ given in \eqref{def:h}.
The polyconvexity of an isotropic function of dimension $d$, thus, requires only the convexification of $h$ acting essentially on a $d$-dimensional manifold in $3$- and $7$-dimensional space for $d=2$ and $d=3$, respectively. This drastic reduction in dimensionality makes the polyconvexification problem feasible even for $d=3$ in many cases.

\subsection{Sketch of the algorithm}
Let $W\colon\R^{d \times d} \to \R_\infty$ be isotropic, let $\Phi\colon\R^d \to \R_\infty$ be the corresponding $\Pid$-invariant function satisfying the conditions \eqref{eq:W=Phi}--\eqref{eq:Phi=W} and let $\hat{F}\in\R^{d\times d}$ have signed singular values $\hat{\nu} = \nu(\hat{F}) \in \R^{d}$. This section presents an abstract algorithm for approximating the lower semicontinuous polyconvex envelope $\Phi^{\lscpc}$ evaluated at the point $\hat{\nu}$, thus providing an approximation of the lower semicontinuous polyconvex envelope $W^{\lscpc}$ in $\hat{F}$ via $W^{\lscpc}(\hat{F})= \Phi^{\lscsvpc}(\hat{\nu})$. 
Since our algorithmic realization already ensures the lower semicontinuity, we abuse the notation by identifying $\Phi^{\pc}$ with $\Phi^{\lscsvpc}$ and $W^{\pc}$ with $W^{\lscpc}$ in the following.

As with any practical algorithm for computing (semi-)convex envelopes, we assume that the set of non-polyconvexity of $\Phi$, i.e., the support of $\Phi-\Phi^{\pc}$, is bounded. Without loss of generality, we assume that $$\supp (\Phi-\Phi^{\pc})\subset [-r,r]^d\eqqcolon B(r)$$ for some bounding box $B(r)$ with radius $r>0$. The bounding box is discretized by some grid. Although more general choices are possible, we restrict ourselves to equidistant lattices of the form 
$$\Sigma_\delta = \delta\,\Z^{d} \cap [-r,r]^d$$ with lattice size $\delta>0$. By $N_\delta:=(2\lfloor r\delta^{-1}\rfloor+1)^d$, we denote the total number of lattice points.

Given the function $\Phi$, the point $\hat\nu$, at which the polyconvex envelope is to be approximated, and the discretization parameters $\delta$ and $r$, the \emph{signed singular value polyconvexification} outlined in Algorithm~\ref{alg:SVP} consists of four main steps. 
\renewcommand{\algorithmicrequire}{\textbf{Input:}}
\renewcommand{\algorithmicensure}{\textbf{Output:}}
\begin{algorithm}
	\begin{algorithmic}[1]				
		\Require{$\Phi, \hat{\nu}, \delta, r$}
		\State{$\Sigma_\delta \coloneqq \delta\, \Z \cap B(r)$ \hfill(generate lattice in bounding box)}
		\State{$X_\delta \coloneqq \m(\Sigma_\delta)$, $h_\delta \coloneqq \Phi(\Sigma_\delta)$ \hfill(evaluate  minors and function)}\label{alg:line:lift}
		\State{$h^{\c}_\delta$ = \texttt{convexify}$([X_\delta, h_\delta])$ \hfill(approximate convex envelope of $h$)}\label{alg:line:convexify}
		\State{$\Phi^{\pc}_\delta(\hat{\nu})$ = \texttt{interpolate}$(h^{\c}_\delta, \m(\hat{\nu}))$ \hfill(evaluate approximate polyconvex envelope at $\hat{\nu}$)}\label{alg:line:interpolate}
		\Ensure{$\Phi^{\pc}_\delta(\hat{\nu})$}
	\end{algorithmic}
	\caption{Signed singular value polyconvexification}
	\label{alg:SVP}
\end{algorithm}

\emph{Step~$1$} of Algorithm~\ref{alg:SVP} represents the generation of the lattice $\Sigma_\delta$ introduced above. In an actual implementation, $\Sigma_\delta$ can be represented by an $N_\delta\times d$ matrix, where the rows contain the coordinates of the lattice points and induce a natural enumeration of the lattice points.

\emph{Step~$2$} of Algorithm~\ref{alg:SVP} lifts the $N_\delta$ lattice points $\Sigma_\delta$ to the points $X_\delta\coloneqq\m(\Sigma_\delta)$ on the manifold $\mathfrak{m}_d$. For $d=2$ and specific choices of $r$ and $\delta$, the resulting points are visualized in Figure~\ref{fig:SignedSingularValueDiscretization}. The lifted points form an $N_\delta\times k_d$ matrix in the implementation. In addition, the function $\Phi$ needs to be evaluated in the lattice points $\Sigma_\delta$, yielding $h_\delta=\Phi(\Sigma_\delta)$, an $N_\delta$-dimensional row vector in an implementation. At this point, possible infinite values of $\Phi$ could (and should in practice) be eliminated from $h_\delta$ as well as the corresponding points/rows from $X_\delta$. 
They do not contribute to the convex envelope. 

\emph{Step~$3$} of Algorithm~\ref{alg:SVP} is the computation of the convex envelope of the points $(h\circ\m)(\Sigma_\delta)$ representing the graph of $h\circ\m$. Computationally, this can be done in several ways. The most well-known algorithm is probably the \texttt{Quickhull} algorithm \cite{BarDobHuh96} originating from the field of computational geometry. This algorithm will be outlined in more detail in Section~\ref{subsec:quickhull} below. Alternatives for computing convex envelopes arising from other fields of mathematics include the computation of the convex envelope by its reformulation as an obstacle problem as done in \cite{Obe07}. Suitable schemes for the resulting nonlinear partial differential equations can be used to approximate the convex envelope. Similarly, the convex envelope of a function $h$ can be computed via a double Legendre--Fenchel conjugation, due to the relation 
\begin{align*}
	h^{\lscc} = h^{**}.
\end{align*}
Algorithmic realizations of the dual Legendre-Fenchel conjugation of a discrete function \(h_\delta\) have been presented, for example, in \cite{Luc96, Luc97, ConErnVer15}. Such approaches can benefit from parallelization, but the choice of the dual lattice is often tricky. Instead of approximating the full envelope in the bounding box, there are direct characterizations of its evaluation at $\hat{\nu}$ in terms of a linear program, as originally proposed by Bartels in the context of polyconvexity of general energy densities in \cite{Bar05}. This variant will be discussed in more detail in Section~\ref{subsec:convexlp}. 

In \emph{Step~$4$}, depending on the choice of the method \texttt{convexify} and its specific output representation, the approximate convex envelope of $h$ at the target value $\hat\nu$ must be evaluated to obtain the desired approximation of $\Phi^{\pc}(\hat\nu)$. The values $h_\delta^\c$ typically represent a continuous piecewise affine function on the convex envelope, which can be represented by a simplicial mesh as illustrated in Figure~\ref{fig:SignedSingularValueDiscretization}. The desired approximation of $\Phi^{\pc}(\hat\nu)$ can be realized by evaluating this piecewise affine function, which is achieved by the function \texttt{interpolate} in the algorithm. The practical implementation is discussed below along with the two convexification methods. 

The quality of the approximation of the polyconvex envelope by Algorithm~\ref{alg:SVP} depends on the lattice size $\delta$. It is almost independent of the choice of \texttt{convexify} whose error tolerances can typically be controlled reliably and accurately. The dependence of the error arising from the lattice discretization has already been quantified in \cite{Bar05}. In our parameter setting, the original error estimate 
\begin{align}\label{eq:errorW}
	0 \leq W^{\pc}_{\delta, r} (F) - W^{\pc}(F) \leq 2 \, c_{\mathcal{I}} \,  \delta^{1 + \alpha} \, |W|_{C^{1, \alpha}(B_{r'}(0))}
\end{align}
of \cite[Theorem 9.10]{Bar15} with some interpolation constant $c_{\mathcal{I}}>0$ can be rewritten in the isotropic setting as 
\begin{align}\label{eq:errorPHI}
	0 \leq \Phi^{\pc}_{\delta, r} (\nu) - \Phi^{\pc}(\nu) \leq 2 \, c_{\mathcal{I}} \,  \delta^{1 + \alpha} \, |\Phi|_{C^{1, \alpha}(B_{r'}(0))}.
\end{align}

\begin{figure}
	\centering	
	\includegraphics[width=.48\textwidth]{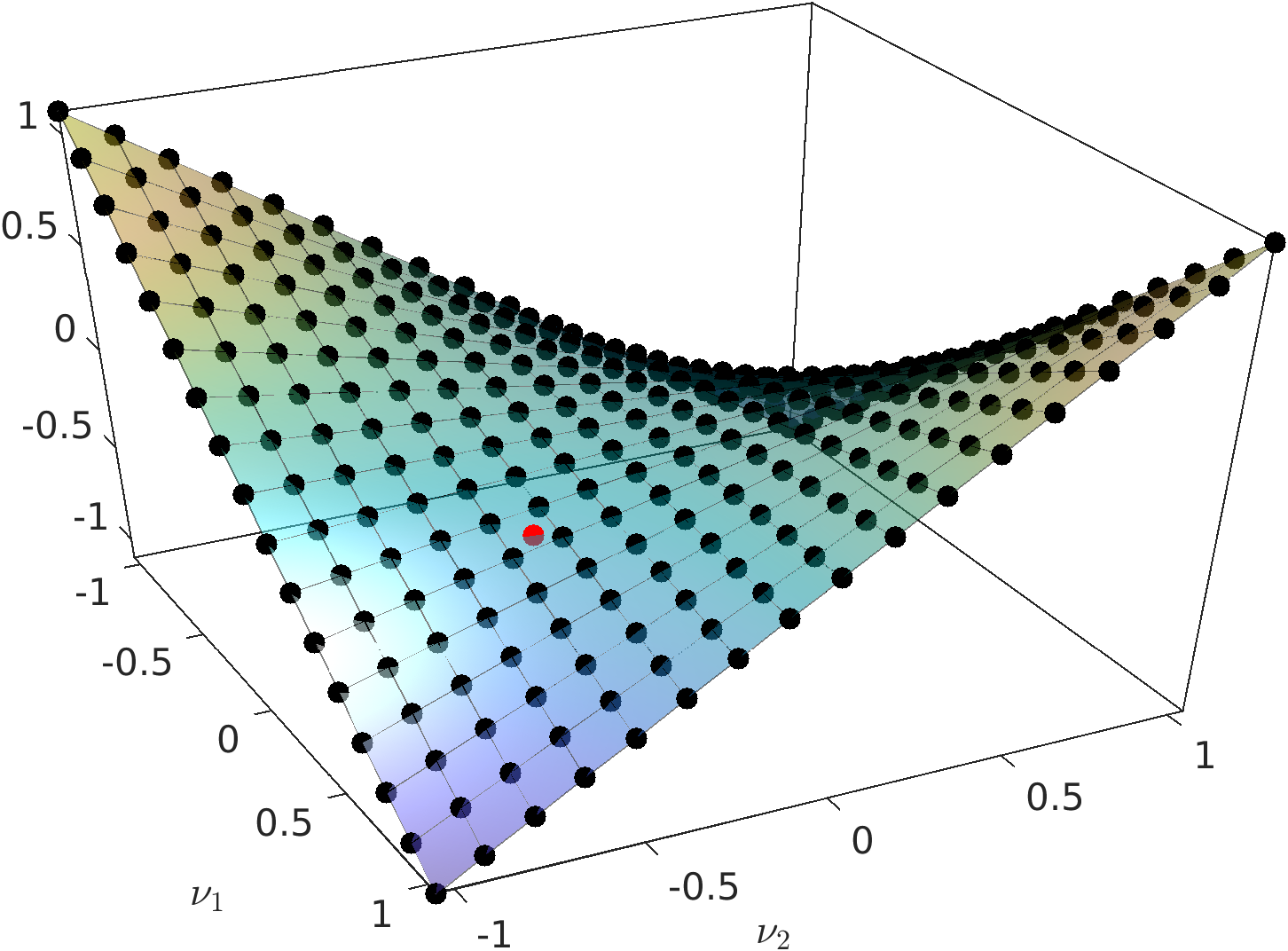}
	\hspace{0.02\textwidth}
	\includegraphics[width=.48\textwidth]{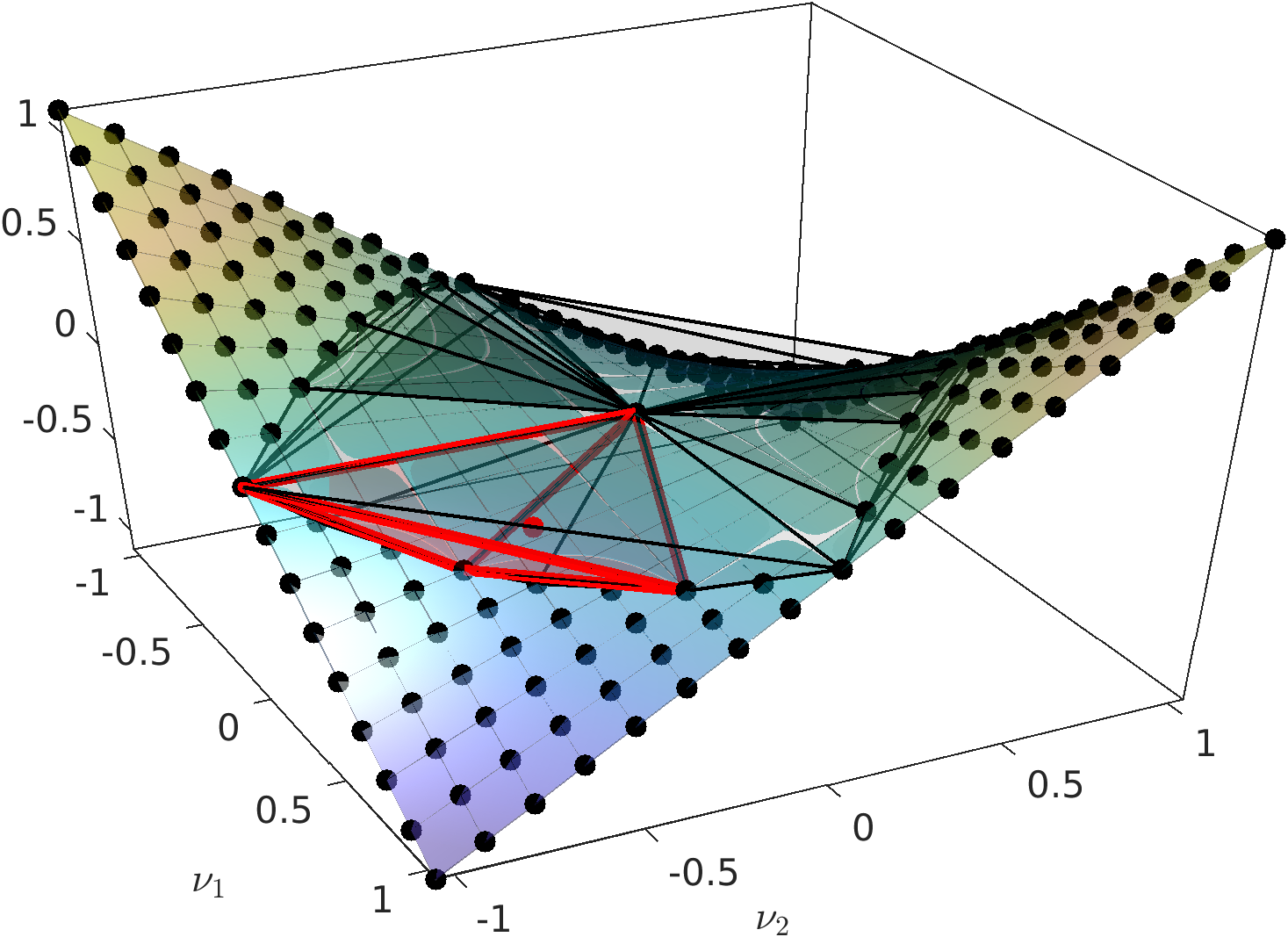}
	\caption{Discretization of signed sigular values: Manifold of minors $\m$ (colored surface, both), lifted lattice points $X_\delta=\m(\Sigma_\delta)$ ($\bullet$, left), supporting points of the polyconvex envelope $X_\delta^\c)$ ($\bullet$, right), simplicial mesh $(\mathcal T_{\delta},X_\delta^\c)$ restricted area of non-polyconvexity (right) for the example of Section~\ref{subsec:KSDexample}. A sample point $\hat{\nu}$ is highlighted in both figures in red. The right figure also highlights the tetrahedron of the mesh that contains $\hat{\nu}$ and thus forms the basis for the interpolatory evaluation of the approximate polyconvex envelope in $\hat\nu$.
	\label{fig:SignedSingularValueDiscretization}}
\end{figure}

\subsection{Convexification by \texttt{Quickhull}}\label{subsec:quickhull}
The method \texttt{convexify} in Step~$3$ of Algorithm~\ref{alg:SVP} computes the lower convex envelope of the rows of the $N_\delta\times (k_d+1)$-matrix $p = [X_\delta, \Phi(\Sigma_\delta)]$. 
This computation can be translated into the geometrical framework and identified with the computation of the convex hull of the corresponding point set in $\R^{k_d+1}$. This issue can be addressed by a computational geometry approach; to this end one can us the \texttt{Quickhull} algorithm \cite{BarDobHuh96}. 
For point sets in dimension higher than three as in our case, it is also the most efficient one known. The algorithm follows a divide-and-conquer approach and is of complexity $\mathcal{O}(M^{\lfloor n / 2\rfloor})$, where $M$ denotes the number of input points and $n$ their dimension \cite{Sei81, Cha93}. This translates to a worst-case complexity of $\mathcal{O}(N_\delta^{\lfloor (k_d+1) / 2\rfloor})$ in the present application. 

The typical output representation of \texttt{Quickhull} is a simplicial mesh $\mathcal T^\c_\delta$ of the $k_d$-dimensional convex hull in a $k_d+1$-dimensional space with vertices given by the subset of the input points. 
This simplicial representation may not be unique unless the vertices are in general position. However, this possible non-uniqueness in the output does not affect the convex hull. Any simplicial representation of the lower hull is perfectly fine for our purposes. 
Another technical problem, which is associated with this computational geometry approach, is that it does not treat the point set as the graph of a function, i.e., it returns the full convex hull of the point cloud rather than the convex envelope which correspond to the function.
However, the convex envelope of the function can be easily extracted from the hull by computing the outer normals of the facets (which are often provided by \texttt{Quickhull} implementations anyway). If the $(k_d+1)$st component of the normal is strictly negative, the corresponding facet is part of the envelope of the function; otherwise, it needs to be removed from the mesh.

The simplicial mesh $\mathcal T^\c_\delta$ usually encodes the convex hull as a surface consisting of facets ($k_d$-simplices) by providing for each facet its $k_d+1$ vertices as row indices of the list of supporting vertices $[X^\c_\delta,h^\c_\delta]$. The restriction of this mesh representation of the lower convex hull to the $k_d$-dimensional space is a simplicial volume mesh $(\mathcal T^\c_\delta,X^\c_\delta)$ of the convex hull of the lifted lattice points $X_\delta^c\subset\mathfrak{m}_d$, as illustrated in Figure~\ref{fig:SignedSingularValueDiscretization}. Thus, the values $h_\delta^\c$ represent a continuous piecewise linear function on this mesh of $\operatorname{conv}(X_\delta^c)$. In step $4$ of the algorithm, the evaluation of this piecewise affine function by \texttt{interpolate} can be realized as follows. A suitable search algorithm returns a simplex $\operatorname{conv}\{x_1,\ldots,x_{k_d}\}\in\mathcal T_\delta$ that contains $\m(\hat\nu)$ (the highlighted simplex in Figure~\ref{fig:SignedSingularValueDiscretization}). It remains to compute the barycentric coordinates $c_1,\ldots,c_{k_d}\geq 0$ of $\m(\hat\nu)$ within this simplex, which satisfy
\begin{equation}
	\label{eq:bary}
	\m(\hat\nu)=\sum_{i=1}^{k_d}\xi_i x_i,\quad \xi_1+\ldots+\xi_{k_d}=1.
\end{equation}
Then the desired approximation of $\Phi^{\pc}(\hat\nu)$ is given by 
$$\Phi_\delta^{\pc}(\hat\nu) = \sum_{i=1}^{k_d}\xi_i(h^\c_\delta)_i.$$
In the case that $\Phi$ attains $\infty$ and some grid points are removed from the discretization, $\hat{\nu}$ is not necessarily contained in a simplex of $\mathcal{T}_\delta$.
This corresponds to the fact that $\Phi_{\delta}^{\pc}(\hat{\nu}) = \infty$ and this aspect has to be taken into account in the evaluation.

It is clear that the approximate polyconvex envelope can be evaluated for multiple arguments by repeated calls of the interpolation procedure but without another call to \texttt{Quickhull}. The feasibility of the resulting practical algorithm is demonstrated in the numerical experiments of Section~\ref{subsec:KSDexample}. 

\subsection{Convexification by linear programming}\label{subsec:convexlp}
The computational geometry approach above aims for the lower convex envelope in the full bounding box. Often, we are only interested in the evaluation of an approximated envelope in a small number of points. In this case, it is useful to use the pointwise characterization of $\Phi^{\pc}$ and $h^{\c}$ at $\hat\nu$ and $\hat x =\m(\hat\nu)$, respectively, which is given by the optimization problem
\begin{align}\label{eq:poly-opt-prob-iso}
	\Phi^{\pc}(\hat\nu) = h^{\c} (\hat x) = \inf \left\{\sum_{i = 1}^{k_d + 1} \xi_i \, h(x_i)\;\biggl\vert\;  \xi_{i} \in [0, 1],\, x_i \in \R^{k_d},\, \sum_{i = 1}^{k_d + 1} \xi_{i} = 1,\, \sum_{i = 1}^{k_d + 1} \xi_i x_i = \hat x \right\}.
\end{align}
This formulation was derived in \cite{Bar05} (see also \cite{EBG13}, \cite{BEG15}, and \cite{Bar15}) for the classical notion of polyconvexity and non-isotropic $W$.
After suitable discretization of $\R^{k_d}$, e.g., by the lifted lattice ${\m(\Sigma_\delta)=X_\delta=\{x_i\}_{i=1}^{N_\delta}}$, this nonlinear optimization problem turns into the following linear program
\begin{align}\label{eq:poly-opt-prob-iso-discr}
	\Phi^{\pc}_\delta (\hat\nu) =h_{\delta}^{\c} (\hat x) = \min \left\{\sum_{i = 1}^{N_\delta} \xi_{i} \, h(x_i) \;\biggl\vert\; \xi_{i} \geq 0,\, \sum_{i = 1}^{N_\delta} \xi_{i} = 1,\, \sum_{i = 1}^{N_\delta} \xi_{i} x_{i} = \hat x \right\}.
\end{align}
Possible infinite values of $h(x_i)=h(\m(\nu_i))=\Phi(\nu_i)$ should be avoided by removing the corresponding vertices from the lattice as discussed earlier. 
Rather than first computing the full lower convex envelope and then searching for a facet that contains $\hat x=\m(\hat\nu)$ as in Figure~\ref{fig:SignedSingularValueDiscretization}, this approach directly computes that facet. In exact arithmetics, there exist a minimizer $\xi$ with at most $k_d+1$ non-zero entries that represents the barycentric coordinates of $\hat x$ with respect to the vertices of the $k_d$-simplex. Up to the possible non-uniqueness of the $k_d+1$ indices, which corresponds to non-zero entries, and the simplicial mesh representation of the convex envelope, these vertices  $x_i$ and the corresponding nonzero coefficients $\xi_i$ coincide with those of \eqref{eq:bary}.
The linear program \eqref{eq:poly-opt-prob-iso-discr} can be solved fairly efficiently by standard algorithms for linear programming available in various software libraries. In \cite{Bar05}, an active set strategy along with multilevel optimization and adaptive techniques was suggested to improve the runtime in the setting of general non-isotropic functions further. 
We will use the algorithm along with the Matlab function for linear programming in the numerical experiments of Section~\ref{sec:examples}.

\section{Numerical experiments}\label{sec:examples}
In this section, we illustrate the performance of the presented singular value polyconvexification algorithm with a number of numerical examples.
Due to the requirement of isotropy, we restrict ourselves to functions that satisfy this condition. 
In our experiments, we use Algorithm~\ref{alg:SVP} with both \texttt{Quickhull} and the linear programming approach to convexification discussed in Sections~\ref{subsec:quickhull} and \ref{subsec:convexlp}, respectively. 
Further, we apply both approaches also for the complete matrix case by convexifying \eqref{eq:H} in order to investigate the computational speedup resulting from the dimension reduction.

A basic Matlab \cite{MATLAB} implementation of the two variants of Algorithm~\ref{alg:SVP} and their application to the examples of Sections~\ref{subsec:KSDexample}--\ref{subsec:doubleWell} is available as supplementary material. 
We use the Matlab internal implementation of the \texttt{Quickhull} algorithm as well as the \texttt{Interior-Point-Legacy} method for solving the linear program \eqref{eq:poly-opt-prob-iso-discr} using \texttt{linprog}.

\subsection{Kohn-Strang-Dolzmann example} \label{subsec:KSDexample}
The following example was studied in \cite{KohStr86a, KohStr86b}, subsequently modified to achieve continuity in \cite{Dol99, DolWal00} and further studied in \cite{Bar05}. 
\begin{figure}
	\begin{center}
		\begin{subfigure}[b]{0.48\textwidth}
			\centering 
			\includegraphics[scale=0.6]{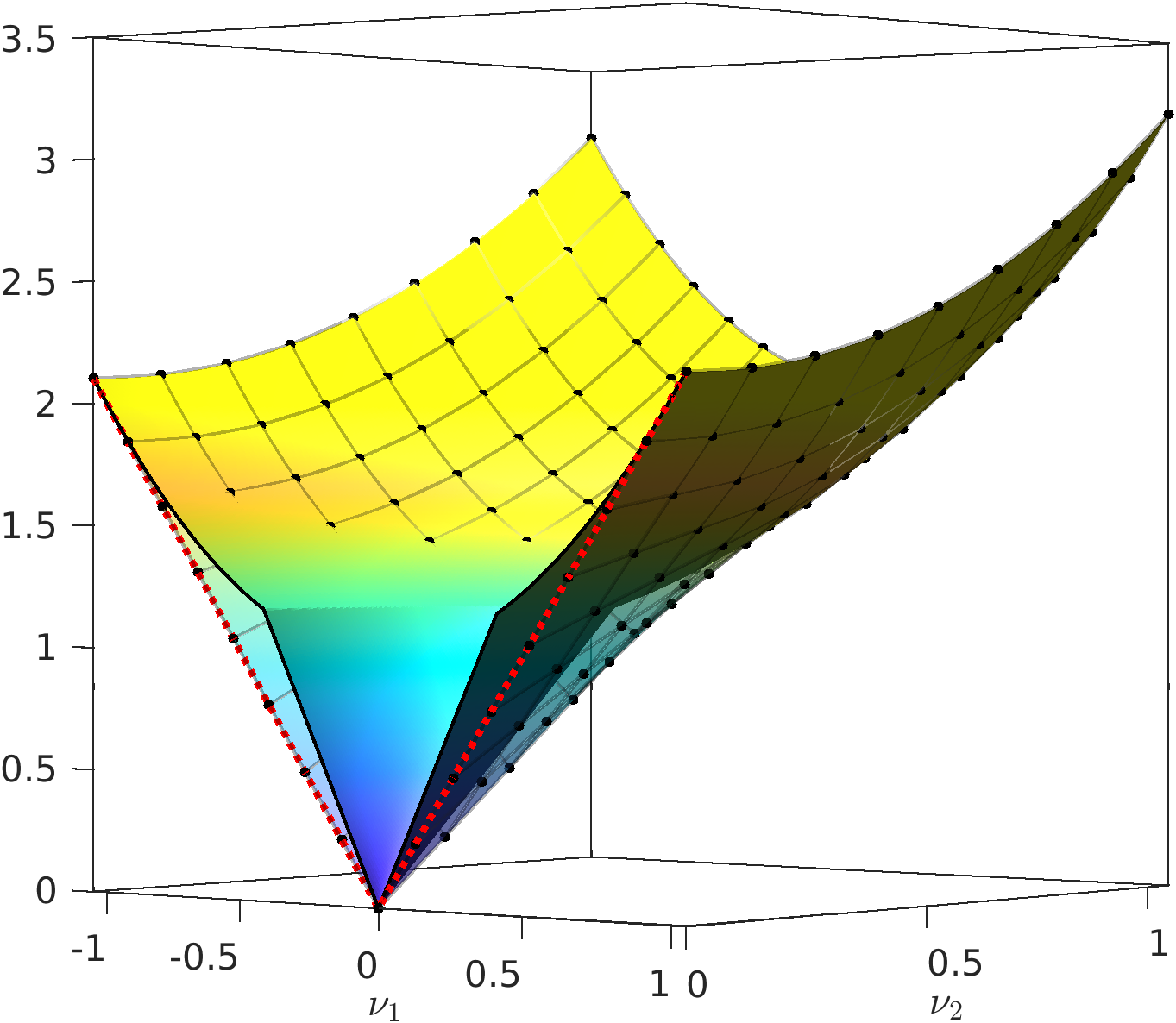}
			\label{fig:phipcKSDol}
			\caption{Surface plots of $\Phi$ (upper) and $\Phi^\pc$ (lower), computed polyconvex hull $\Phi^\pc_\delta$ shown as a lattice with black bullets.}
		\end{subfigure}\hfill
		\begin{subfigure}[b]{0.48\textwidth}
			\centering 
			\includegraphics[scale=0.58]{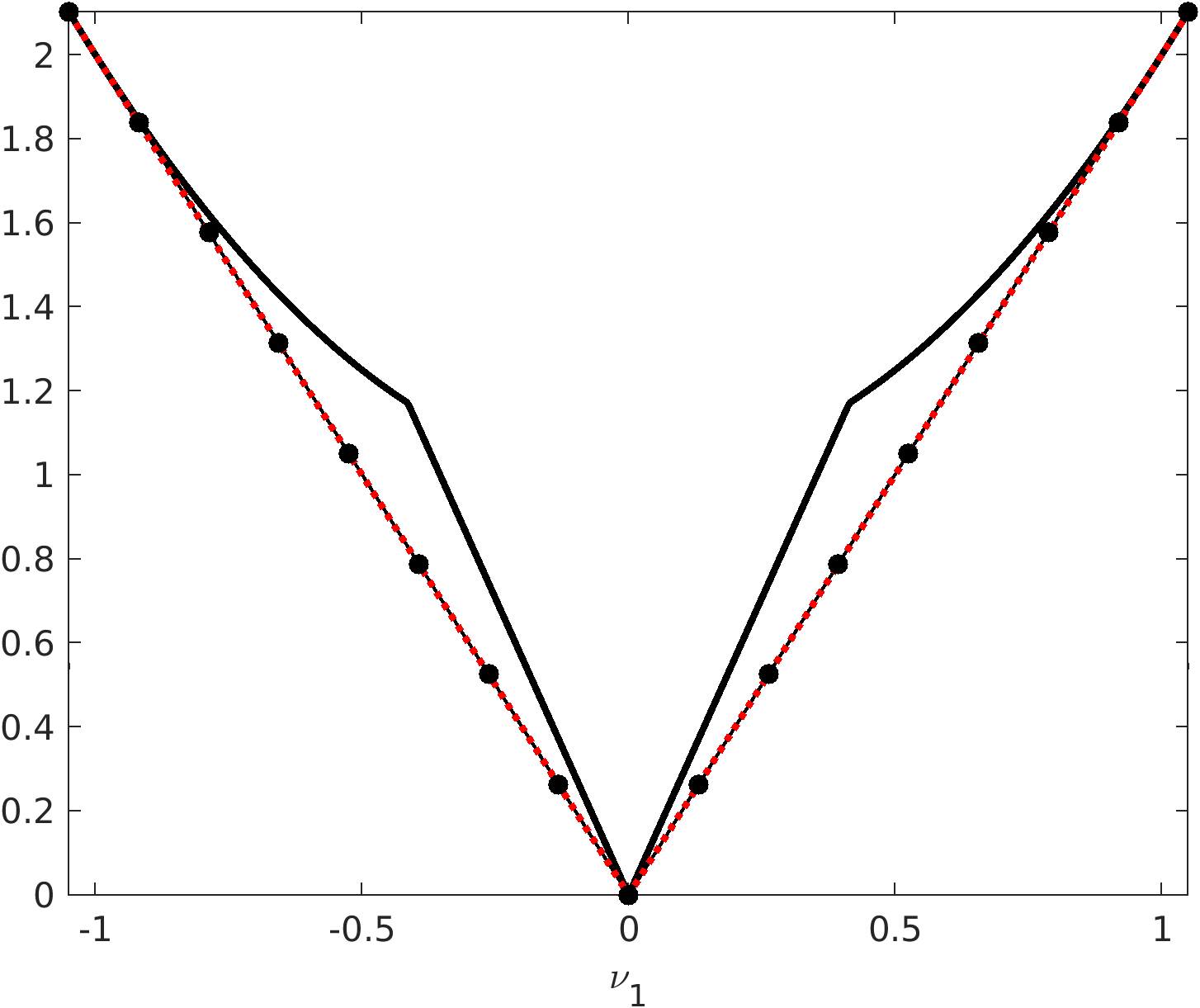}
			\caption{$\Phi$ (solid black) and $\Phi^\pc$ (dotted red) vs. $\nu_1$ for $\nu_2=0$, computed polyconvex hull $\Phi^\pc_\delta$ shown as a 1d lattice with black bullets.}
			\label{fig:phiKSDol}
		\end{subfigure}
		\caption{Illustration of function $\Phi$ from \eqref{eq:PHI}, its polyconvex hull $\Phi^\pc$ from \eqref{eq:PHIpc} and computed polyconvex hull $\Phi^\pc_\delta$ using \texttt{Quickhull} and discretization parameters $r = 1.05$, $\delta = 0.13125$, using $N_\delta = 17^2$ lattice points.}
		\label{fig:functionsKSDol}
	\end{center}
\end{figure}
We consider the function ${W\colon \R^{2 \times 2} \to \R}$, defined as
\begin{align*}
W(F) \coloneqq
\begin{cases}
	1 + |F|^2 & \text{if } |F| \geq \sqrt{2} - 1, \\
	2\, \sqrt{2} \, |F| & \text{if } |F| \leq \sqrt{2} - 1,
\end{cases}
\end{align*}
where $|F| \coloneqq\big(\sum_{i,j=1}^{d} F_{ij}^2\big)^{1/2}$ denotes the Frobenius norm of $F$. The polyconvex hull of $W$ is explicitly known and reads
\begin{align*}
W^{\pc} (F) =
\begin{cases}
1 + |F|^2 & \text{if } \varrho \geq 1, \\
2 \, (\varrho(F) - |\det(F)|) & \text{if } \varrho \leq 1,
\end{cases}
\end{align*}
where $\varrho(F) \coloneqq \sqrt{|F| + 2 \, |\det F|}$. 
The functions $W$ and $W^{\pc}$ are isotropic and, rewritten in terms of the signed singular values, they reduce to $\Phi,\Phi^{\pc} \colon \R^{2} \to \R$ with
\begin{align}\label{eq:PHI}
	\Phi(\nu) \coloneqq 
	\begin{cases}
		1 + \nu_1^2 + \nu_2^2 & \text{if } \sqrt{\nu_1^2 + \nu_2^2} \geq \sqrt{2} - 1, \\
		2\, \sqrt{2} \, \sqrt{\nu_1^2 + \nu_2^2}  & \text{if } \sqrt{\nu_1^2 + \nu_2^2} \leq \sqrt{2} - 1,
	\end{cases}
\end{align}
and
\begin{align}\label{eq:PHIpc}
	\Phi^{\pc} (F) =
	\begin{cases}
		1 + \nu_1^2 + \nu_2^2 & \text{if } \varrho \geq 1, \\
		2 \, (\varrho(\nu_1 \, \nu_2) - |\nu_1 \, \nu_2|) & \text{if } \varrho \leq 1,
	\end{cases}
\end{align}
with $\varrho(\nu)^2 = \sqrt{\nu_1^2 + \nu_2^2} + 2 \, |\nu_1 \, \nu_2|$. 
Figure~\ref{fig:functionsKSDol} shows the function $\Phi$, its polyconvex envelope $\Phi^\pc$ and the computed approximation $\Phi^{\pc}_\delta$, resulting from the Algorithm~\ref{alg:SVP} with \texttt{Quickhull} as outlined in Section~\ref{subsec:quickhull}.
In Figure~\ref{fig:errorsAndRunntimesKSDol}, we compare this variant of the algorithm (referred to as svpc QH) with the variant based on linear programming (svpc LP) described in Section~\ref{subsec:convexlp}. For reference, we also show the results for the corresponding algorithms for the polyconvex hull $W$ without exploiting isotropy. As a linear programming variant of the polyconvexification of $W$, we use the adaptive algorithm of \cite{Bar05} and its implementation presented in \cite{Bar15} (denoted pc \cite{Bar05}). We also implemented the direct convexification of $W$ using \texttt{Quickhull}, which convexifies $H$ in \eqref{eq:Wpc=HcM} (pc QH). All simulations were performed on a state-of-the-art laptop computer.

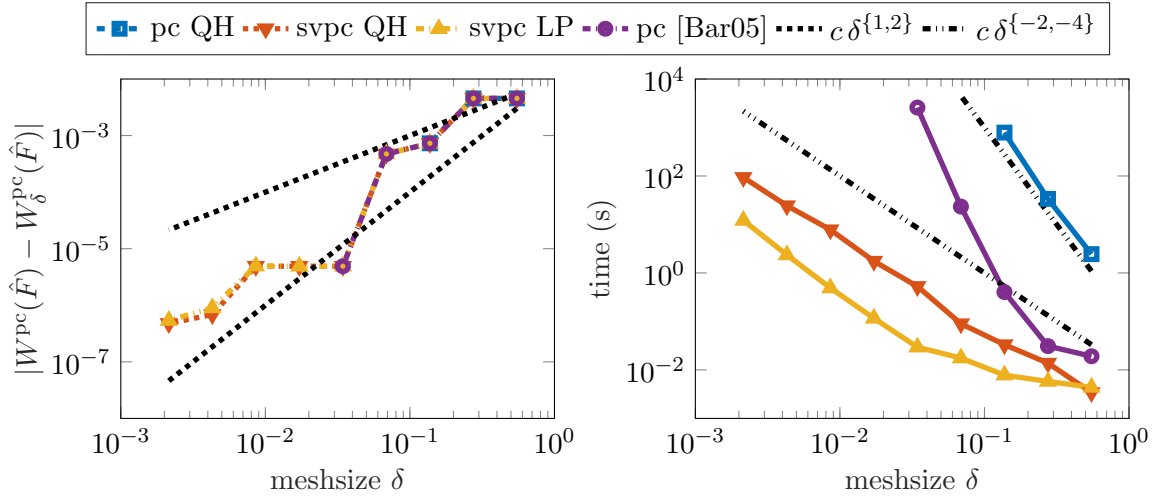
\begin{figure}
	\begin{center}
		\setlength\figurewidth{0.4\textwidth}
%
%
\definecolor{mycolor1}{rgb}{0.00000,0.44700,0.74100}%
\definecolor{mycolor2}{rgb}{0.85000,0.32500,0.09800}%
\definecolor{mycolor3}{rgb}{0.92900,0.69400,0.12500}%
\definecolor{mycolor4}{rgb}{0.49400,0.18400,0.55600}%
\begin{tikzpicture}

\begin{axis}[%
width=0.951\figurewidth,
height=0.75\figurewidth,
at={(-0.63\figurewidth,0\figurewidth)},
scale only axis,
xmode=log,
xmin=0.001,
xmax=1,
xminorticks=true,
xlabel style={font=\color{white!15!black}},
xlabel={meshsize $\delta$},
ymode=log,
ymin=1e-08,
ymax=0.01,
yminorticks=true,
ylabel style={font=\color{white!15!black}},
ylabel={$|W^{\pc}(\hat{F}) - W_\delta^{\pc}(\hat{F})|$},
axis background/.style={fill=white},
legend columns = 7,
legend style={legend cell align=left, align=left, at={(1.1,1.07)}, anchor=south, draw=white!15!black}
]

\addplot [color=mycolor1, dashed, line width=2.0pt, mark=square, mark options={solid, mycolor1}]
  table[row sep=crcr]{%
0.55	0.00454545454545474\\
0.275	0.00454545454545463\\
0.1375	0.000730519480519831\\
0.06875	0\\
0.034375	0\\
0.0171875	0\\
0.00859375	0\\
0.004296875	0\\
0.0021484375	0\\
};
\addlegendentry{pc QH}

\addplot [color=mycolor2, dotted, line width=2.0pt, mark=triangle, mark options={solid, rotate=180, mycolor2}]
  table[row sep=crcr]{%
0.55	0.00454545454545452\\
0.275	0.00454545454545463\\
0.1375	0.000730519480519942\\
0.06875	0.000473484848485195\\
0.034375	4.89811912274263e-06\\
0.0171875	4.89811912240956e-06\\
0.00859375	4.89811912240956e-06\\
0.004296875	6.85841787140262e-07\\
0.0021484375	4.77303274792895e-07\\
};
\addlegendentry{svpc QH}

\addplot [color=mycolor3, dash dot dot, line width=2.0pt, mark=triangle, mark options={solid, mycolor3}]
  table[row sep=crcr]{%
0.55	0.00454546520632226\\
0.275	0.00454545656076755\\
0.1375	0.000730533314592496\\
0.06875	0.000473485173178023\\
0.034375	4.90284235410421e-06\\
0.0171875	4.89816805027132e-06\\
0.00859375	4.92141139474267e-06\\
0.004296875	8.98402114257735e-07\\
0.0021484375	5.38319854359592e-07\\
};
\addlegendentry{svpc LP}

\addplot [color=mycolor4, dashdotted, line width=2.0pt, mark=o, mark options={solid, mycolor4}]
  table[row sep=crcr]{%
0.55	0.00454545454545463\\
0.275	0.00454545454545463\\
0.1375	0.000730519480520608\\
0.06875	0.00047348484848464\\
0.034375	0.000004898119126\\
0.0171875	0\\
0.00859375	0\\
0.004296875	0\\
0.0021484375	0\\
};
\addlegendentry{pc [Bar05]}

\addplot [color=black, dotted, line width=2.0pt]
  table[row sep=crcr]{%
0.55	0.0055\\
0.275	0.00275\\
0.1375	0.001375\\
0.06875	0.0006875\\
0.034375	0.00034375\\
0.0171875	0.000171875\\
0.00859375	8.59375e-05\\
0.004296875	4.296875e-05\\
0.0021484375	2.1484375e-05\\
};
\addlegendentry{$c\,\delta^{\{1, 2\}}$}

\addplot [color=black, dash dot dot, line width=2.0pt]
table[row sep=crcr]{%
	NaN	NaN\\
};
\addlegendentry{$c\,\delta^{\{-2, -4\}}$}
\addlegendimage{color=black, dash dot dot, line width=2.0pt}

\addplot [color=black, dotted, line width=2.0pt, forget plot]
  table[row sep=crcr]{%
0.55	0.003025\\
0.275	0.00075625\\
0.1375	0.0001890625\\
0.06875	4.7265625e-05\\
0.034375	1.181640625e-05\\
0.0171875	2.9541015625e-06\\
0.00859375	7.38525390625e-07\\
0.004296875	1.8463134765625e-07\\
0.0021484375	4.61578369140625e-08\\
};
\end{axis}

\begin{axis}[%
width=0.951\figurewidth,
height=0.75\figurewidth,
at={(0.63\figurewidth,0\figurewidth)},
scale only axis,
xmode=log,
xmin=0.001,
xmax=1,
xminorticks=true,
xlabel style={font=\color{white!15!black}},
xlabel={meshsize $\delta$},
ymode=log,
ymin=0.001,
ymax=10000,
yminorticks=true,
ylabel style={font=\color{white!15!black}},
ylabel={time (s)},
axis background/.style={fill=white},
legend style={at={(0.03,0.83)}, anchor=south west, legend cell align=left, align=left, draw=white!15!black}
]
\addplot [color=mycolor1, line width=2.0pt, mark=square, mark options={solid, mycolor1}]
table[row sep=crcr]{%
	0.55	2.443527\\
	0.275	33.721858\\
	0.1375	785.008479\\
	0.06875	0\\
	0.034375	0\\
	0.0171875	0\\
	0.00859375	0\\
	0.004296875	0\\
	0.0021484375	0\\
};

\addplot [color=mycolor2, line width=2.0pt, mark=triangle, mark options={solid, rotate=180, mycolor2}]
table[row sep=crcr]{%
	0.55	0.003421\\
	0.275	0.013814\\
	0.1375	0.033209\\
	0.06875	0.089796\\
	0.034375	0.52757\\
	0.0171875	1.749567\\
	0.00859375	7.786541\\
	0.004296875	24.386988\\
	0.0021484375	93.43147\\
};

\addplot [color=mycolor3, line width=2.0pt, mark=triangle, mark options={solid, mycolor3}]
table[row sep=crcr]{%
	0.55	0.004378\\
	0.275	0.005780\\
	0.1375	0.00781\\
	0.06875	0.017568\\
	0.034375	0.02952\\
	0.0171875	0.116067\\
	0.00859375	0.490234\\
	0.004296875	2.357029\\
	0.0021484375	12.103715\\
};

\addplot [color=mycolor4, line width=2.0pt, mark=o, mark options={solid, mycolor4}]
table[row sep=crcr]{%
	0.55	0.019138\\
	0.275	0.030983\\
	0.1375	0.40375\\
	0.06875	23.361705\\
	0.034375	2576.705955\\
	0.0171875	0\\
	0.00859375	0\\
	0.004296875	0\\
	0.0021484375	0\\
};

\addplot [color=black, dash dot dot, line width=2.0pt]
table[row sep=crcr]{%
	0.55	0.0330578512396694\\
	0.275	0.132231404958678\\
	0.1375	0.528925619834711\\
	0.06875	2.11570247933884\\
	0.034375	8.46280991735537\\
	0.0171875	33.8512396694215\\
	0.00859375	135.404958677686\\
	0.004296875	541.619834710744\\
	0.0021484375	2166.47933884297\\
};

\addplot [color=black, dash dot dot, line width=2.0pt, forget plot]
table[row sep=crcr]{%
	0.55	1.09282152858411\\
	0.275	17.4851444573458\\
	0.1375	279.762311317533\\
	0.06875	4476.19698108053\\
};
\end{axis}
\end{tikzpicture}%
		\caption{Quantitative comparison of computational polyconvexification of Kohn-Strang-Dolzmann example of Section~\ref{subsec:KSDexample} in the matrix $\hat F $ and known exact values $W(\hat{F}) \approx 1.095$ and $W^{\pc}(\hat{F}) = 0.9$. Left: Absolute error with respect to lattice parameter $\delta$ for several methods. Right: Corresponding computing times.}
		\label{fig:errorsAndRunntimesKSDol}
	\end{center}
\end{figure}
The left graph in Figure~\ref{fig:errorsAndRunntimesKSDol} shows absolute errors in evaluating the polyconvex hull at the point 
$$\hat F=\begin{bmatrix} 0.2 & 0.1 \\ 0.1 & 0.3 \end{bmatrix}$$ 
with singular values $\hat\nu_1 \approx 0.3618$ and $\hat\nu_2 \approx 0.1282$. For fixed $r=1.1$, the presented algorithms indeed converge at a rate $\delta^{1 + \alpha}$ for some $\alpha \in [0, 1]$ as the lattice parameter $\delta$ is refined, in agreement with theoretical predictions \eqref{eq:errorW}--\eqref{eq:errorPHI}.

Although the algorithms are virtually equally accurate, they differ significantly in computational complexity. With the exception of the pc \cite{Bar05} variant, computation times actually scale only linearly with the number of grid points, which in turn is proportional to $\delta^{-4}$ for pc QH and $\delta^{-2}$ for svpc QH and svpc LP. This is much better than the worst-case complexity of \texttt{Quickhull}, which predicts $\delta^{-8}$ and $\delta^{-4}$ for pc QH. The observed behavior seems to correspond to the \texttt{Quickhull} complexity in the dimensional space $4$ or $2$, which represents the dimensions of the manifolds and not the ambient space.

In any case, the numerical results clearly demonstrate the claimed superiority of the new algorithms in the isotropic regime. Of the two variants of singular value polyconvexity, the linear programming variant appears to be faster for our implementation. Note, however, that \texttt{Quickhull} approximates the envelope throughout the bounding box and it would easily pay off if the polyconvex envelope is to be evaluated at multiple points. 

\subsection{Multi-dimensional double-well potential}\label{subsec:doubleWell} 
Since the generalization of the previous example to three dimensions is not known explicitly, we consider the function
\begin{align*}
	W \colon \R^{d \times d} & \to \R \\
	F & \mapsto \left(|F|^2 - 1\right)^2 
\end{align*}
which is an established benchmark for analytical and computational semi-convexification \cite{KohStr86a, DolWal00, Bar05}. For any $d\in\{1,2,3\}$, the rank-one, the quasiconvex and the polyconvex envelope of $W$ coincide with the convex envelope given by
\begin{align*}
	W^{\pc}(F) =
	\begin{cases}
		\left(|F|^2 - 1\right)^2 & \text{if } |F| \geq 1,\\
		0 & \text{else.}
	\end{cases}
\end{align*}
The function $W$ is isotropic and can reformulated in terms of the signed singular values via the function $\Phi$ by
\begin{align} \label{eq:PhiDoubleWell}
	\Phi (\nu) = \left(\sum_{i=1}^{d} \nu_i^2 - 1 \right)^2.
\end{align}
Similarly, the polyconvex hull of $\Phi$ can be expressed as 
\begin{align} \label{eq:PHIpcDoubleWell}
	\Phi^{\pc} (\nu) = \begin{cases}
		\left(\sum_{i=1}^{d} \nu_i^2  - 1 \right)^2 & \text{if } \sum_{i=1}^{d} \nu_i^2 \geq 1,\\
		0 & \text{else.}
	\end{cases}
\end{align}
An illustration of the function $\Phi$ and the computed envelope $\Phi_{\delta}^{\pc}$ for three spatial dimensions is given in Figure~\ref{fig:doubleWell3DSlice}. 
\begin{figure}
	\begin{center}
		\setlength\figurewidth{0.5\textwidth}
		\input{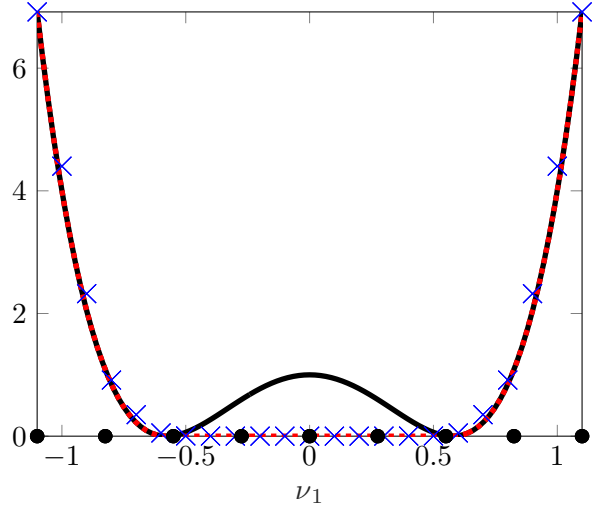}		
		\caption{Slice of the three dimensional double well example. Function $\Phi$ \eqref{eq:PhiDoubleWell} (solid black), its polyconvex hull $\Phi^\pc$ \eqref{eq:PHIpcDoubleWell} (dotted red) vs. $\nu_1=\nu_2=\nu_3$, evaluation of polyconvex hull $\Phi^\pc_\delta$ via svpc LP \eqref{eq:poly-opt-prob-iso-discr} (blue crosses), a selection of lattice points involved in the minimization problem (black bullets).}
		\label{fig:doubleWell3DSlice}
	\end{center}
\end{figure}
The slice along the diagonal direction $\nu_1\,[1,1,1]^T$ is plotted.
The calculation via the svpc LP approach is based on the lattice characterized by $\delta=0.28125$ and the radius $r = 1.125$. In total $9^3 = 729$ lattice points are involved in the minimization problem. A selection of those lattice points is marked by black bullets, exactly the ones lying on the diagonal slice. The sequential pointwise evaluation of $\Phi_{\delta}^{\pc}$ is marked by blue crosses.

Given the performance of the algorithms in the two-dimensional example of Section~\ref{subsec:KSDexample}, we restrict the quantitative study of convergence and complexity to the linear programming variant of Algorithm~\ref{alg:SVP}. For $d=2$ and $d=3$, absolute errors and computing times are shown in Figure~\ref{fig:doubleWellComputingTimes} relative to the lattice parameter $\delta$. The radius of the bounding box is set to $r=2$. For this particularly smooth energy density and polyconvex envelope, the convergence of the error in the points $$\hat{F} = \begin{bmatrix} 0.2 & 0.1 \\ 0.1 & 0.3\end{bmatrix}$$ and $\hat{F} = \diag (0.3, 0.3, 0.3)$ is faster than expected, proportional to $\delta^4$ in both the two- and three-dimensional case. The corresponding computing times scale linearly in the number of lattice points, i.e., they are proportional to $\delta^{-d}$ for $d=2,3$. This example shows that the presented algorithm is able to perform a convexification of isotropic functions in the three-dimensional case which is hardly feasible without exploiting isotropy. 
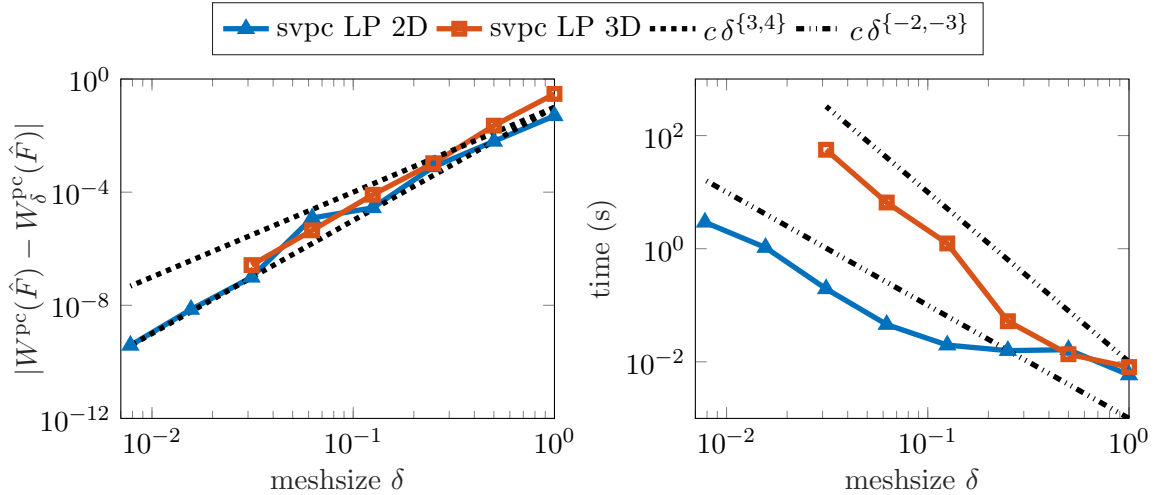
\begin{figure}
	\begin{center}
		\setlength\figurewidth{0.4\textwidth}
%
%
\definecolor{mycolor1}{rgb}{0.00000,0.44700,0.74100}%
\definecolor{mycolor2}{rgb}{0.85000,0.32500,0.09800}%
\begin{tikzpicture}

\begin{axis}[%
width=0.951\figurewidth,
height=0.75\figurewidth,
at={(-0.63\figurewidth,0\figurewidth)},
scale only axis,
xmode=log,
xmin=0.007,
xmax=1,
xminorticks=true,
xlabel style={font=\color{white!15!black}},
xlabel={meshsize $\delta$},
ymode=log,
ymin=1e-12,
ymax=1,
yminorticks=true,
ylabel style={font=\color{white!15!black}},
ylabel={$|W^{\pc}(\hat{F}) - W_\delta^{\pc}(\hat{F})|$},
axis background/.style={fill=white},
legend columns = 4,
legend style={legend cell align=left, align=left, at={(1.1,1.07)}, anchor=south, draw=white!15!black}
]
\addplot [color=mycolor1, line width=2.0pt, mark=triangle, mark options={solid, mycolor1}]
  table[row sep=crcr]{%
1	0.0500000018903411\\
0.5	0.00625000147537883\\
0.25	0.000781254151022917\\
0.125	2.79019066323088e-05\\
0.0625	1.22070327453941e-05\\
0.03125	9.76562568372906e-08\\
0.015625	7.10604643916807e-09\\
0.0078125	3.82151400515227e-10\\
};
\addlegendentry{svpc LP 2D}

\addplot [color=mycolor2, line width=2.0pt, mark=square, mark options={solid, mycolor2}]
  table[row sep=crcr]{%
1	0.297000000000004\\
0.5	0.0225000000002624\\
0.25	0.00105468750030426\\
0.125	8.11298078743094e-05\\
0.0625	4.46901483192023e-06\\
0.03125	2.63991008308178e-07\\
};
\addlegendentry{svpc LP 3D}

\addplot [color=black, dotted, line width=2.0pt]
  table[row sep=crcr]{%
1	0.1\\
0.5	0.0125\\
0.25	0.0015625\\
0.125	0.0001953125\\
0.0625	2.44140625e-05\\
0.03125	3.0517578125e-06\\
0.015625	3.814697265625e-07\\
0.0078125	4.76837158203125e-08\\
};
\addlegendentry{$c\,\delta^{\{3, 4\}}$}

\addplot [color=black, dash dot dot, line width=2.0pt]
table[row sep=crcr]{%
	NaN	NaN\\
};
\addlegendentry{$c\,\delta^{\{-2,-3\}}$}
\addlegendimage{color=black, dash dot dot, line width=2.0pt}

\addplot [color=black, dotted, line width=2.0pt, forget plot]
  table[row sep=crcr]{%
1	0.1\\
0.5	0.00625\\
0.25	0.000390625\\
0.125	2.44140625e-05\\
0.0625	1.52587890625e-06\\
0.03125	9.5367431640625e-08\\
0.015625	5.96046447753906e-09\\
0.0078125	3.72529029846191e-10\\
};
\end{axis}

\begin{axis}[%
width=0.951\figurewidth,
height=0.75\figurewidth,
at={(0.63\figurewidth,0\figurewidth)},
scale only axis,
xmode=log,
xmin=0.007,
xmax=1,
xminorticks=true,
xlabel style={font=\color{white!15!black}},
xlabel={meshsize $\delta$},
ymode=log,
ymin=0.001,
ymax=1000,
yminorticks=true,
ylabel style={font=\color{white!15!black}},
ylabel={time (s)},
axis background/.style={fill=white},
legend style={at={(0.03,0.93)}, anchor=south west, legend cell align=left, align=left, draw=white!15!black}
]
\addplot [color=mycolor1, line width=2.0pt, mark=triangle, mark options={solid, mycolor1}]
  table[row sep=crcr]{%
1	0.005857\\
0.5	0.016404\\
0.25	0.015772\\
0.125	0.019777\\
0.0625	0.045687\\
0.03125	0.195253\\
0.015625	1.049023\\
0.0078125	2.935263\\
};

\addplot [color=mycolor2, line width=2.0pt, mark=square, mark options={solid, mycolor2}]
  table[row sep=crcr]{%
1	0.008094\\
0.5	0.013642\\
0.25	0.052252\\
0.125	1.23851\\
0.0625	6.560479\\
0.03125	56.118723\\
};

\addplot [color=black, dash dot dot, line width=2.0pt]
  table[row sep=crcr]{%
1	0.001\\
0.5	0.004\\
0.25	0.016\\
0.125	0.064\\
0.0625	0.256\\
0.03125	1.024\\
0.015625	4.096\\
0.0078125	16.384\\
};

\addplot [color=black, dash dot dot, line width=2.0pt, forget plot]
  table[row sep=crcr]{%
1	0.01\\
0.5	0.08\\
0.25	0.64\\
0.125	5.12\\
0.0625	40.96\\
0.03125	327.68\\
};
\end{axis}
\end{tikzpicture}%
		\caption{Quantitative comparison of computational polyconvexification via the svpc LP approach of the two and three dimensional double well example of Section~\ref{subsec:doubleWell} in the matrix $\hat F $. Left: Absolute error with respect to lattice parameter $\delta$ for $d \in \{2, 3\}$. Right: Corresponding computing times.}	
		\label{fig:doubleWellComputingTimes}
	\end{center}
\end{figure}

\subsection{Exponentiated Hencky-logarithmic energies} 
The new efficient algorithms for polyconvexication of isotropic energies allow us to shed light on a family of exponentiated Hencky type energies $W_{\text{eH}}: \R^{d \times d} \to \R$ recently proposed by  \cite{NefLanGhiMarSte2015}. Given parameters $\mu, \kappa, k,\ell>0$, they are given by
\begin{align*}
	W_{\text{eH}}(F) =
	\begin{cases} 
		\frac{\mu}{k} \, e^{k \|\dev_d \log U\|^2} + \frac{\kappa}{2 \ell} \, e^{\ell [\log \det U]^2}& \text{ if } \det(F) > 0, \\ 
		\infty & \text{ if } \det(F) \leq 0,
	\end{cases}
\end{align*}
where $U \coloneqq \sqrt{F^T F}$, $\log X$ the matrix logarithm and $\dev_d X = X - \frac{1}{d} \tr(X) I_d$ the deviatoric part of a matrix $X \in \R^{d \times d}$.
In the two-dimensional case, \cite[Theorem 3.11]{NefLanGhiMarSte2015} shows that if $k \geq \frac{1}{3}$ and $\ell\geq \frac{1}{8}$ then $W_{\text{eH}}$ is polyconvex. We will study the sharpness of this result using our algorithm. For this purpose, $W_{\text{eH}}$ is rephrased in terms of signed singular values as
\begin{align} \label{eq:PhiHencky}
	\Phi_{\text{eH}}(\nu) =
	\begin{cases} 
		\frac{\mu}{k} \, e^{k \|\diag(\log |\nu_1|, \dots, \log|\nu_d|) - \log(\prod_{i=1}^{d} |\nu_i|) I_d\|^2} + \frac{\kappa}{2 \ell} \, e^{\ell \log^2(\prod_{i=1}^{d} |\nu_i|)}& \text{ if }  \prod_{i=1}^{d} \nu_i > 0, \\ 
		\infty & \text{ else}.
	\end{cases}
\end{align}
Figure~\ref{fig:HenckyEnergy} shows contour plots of $\Phi_{\text{eH}}$ in the non-polyconvex regime and the computed polyconvex envelope $\Phi_{\text{eH}, \delta}^{\pc}$ for parameters $k=0.01, \ell=0.01, \mu = 1$ and $\kappa = 1$. The approximation $\Phi_{\text{eH}, \delta}^{\pc}$ was computed by the \texttt{Quickhull} approach and still shows some non-convexity.
\begin{figure}
	\begin{center}	
	\setlength\figurewidth{0.44\textwidth}
	\includegraphics{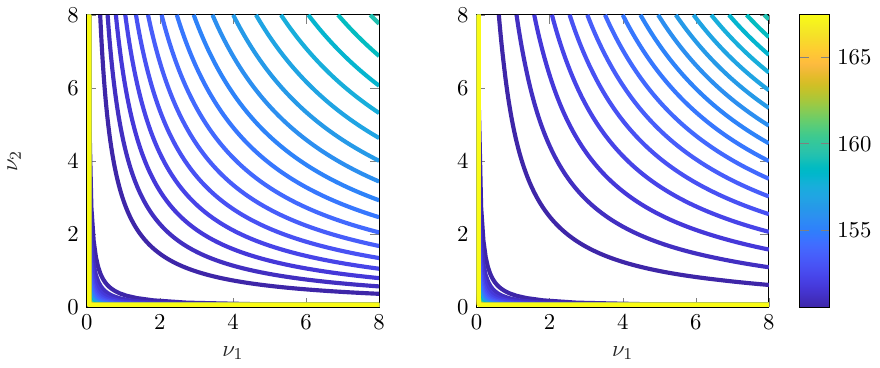}
	\caption{$\Phi_{\text{eH}}$ from \eqref{eq:PhiHencky} in the non polyconvex regime (left) and approximated polyconvex envelope $\Phi_{\text{eH}, \delta}^{\pc}$ by svpc QH approach (right) using lattice parameters $r=8$ and $\delta=0.0625$.}
	\label{fig:HenckyEnergy}
	\end{center}
\end{figure}

We fix the material parameters $\mu = 1$ and $\kappa = 1$ and study the polyconvexity of $\Phi_{\text{eH}}$ depending on the further parameters $k$ and $\ell$.
We choose a rather large bounding box of radius $r=6$ and fix the lattice parameter to $\delta=0.09375$.
We compute the polyconvex envelope in all lattice points using the \texttt{Quickhull} approach and compute the maximal absolute error between this approximate polyconvex hull and the original function. Errors at the order of the lattice parameter indicate polyconvexity of the original function while significantly larger errors indicate non-polyconvexity.
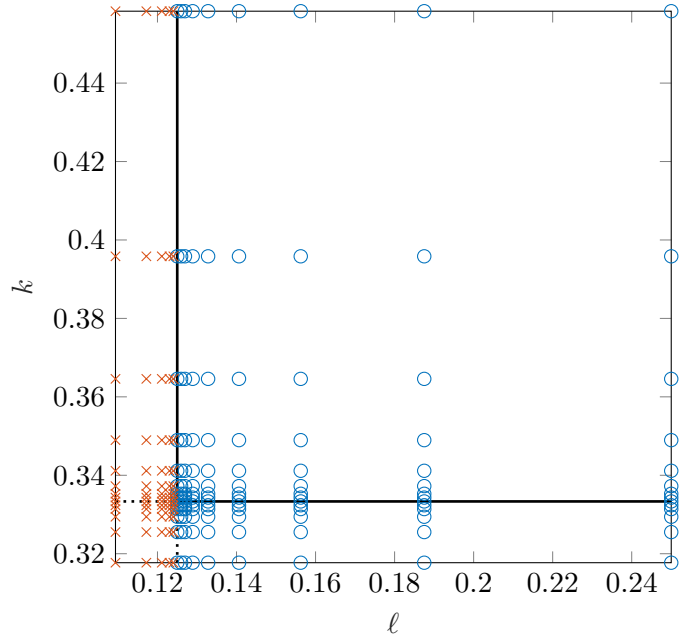
\begin{figure}
	\begin{center}
		\setlength\figurewidth{0.65\textwidth}
%
%
\definecolor{mycolor1}{rgb}{0.00000,0.44700,0.74100}%
\definecolor{mycolor2}{rgb}{0.85000,0.32500,0.09800}%
\begin{tikzpicture}

\begin{axis}[%
width=0.75\figurewidth,
height=0.75\figurewidth,
at={(0\figurewidth,0\figurewidth)},
scale only axis,
xmin=0.109375,
xmax=0.25,
xminorticks=true,
xlabel style={font=\color{white!15!black}},
xlabel={$\ell$},
ymin=0.317708333333333,
ymax=0.458333333333333,
yminorticks=true,
ylabel style={font=\color{white!15!black}},
ylabel={$k$},
axis background/.style={fill=white}
]
\addplot [color=mycolor1, only marks, mark size=2.5pt, mark=o, mark options={solid, mycolor1}, forget plot]
  table[row sep=crcr]{%
0.125	0.317708333333333\\
0.125	0.325520833333333\\
0.125	0.329427083333333\\
0.125	0.331380208333333\\
0.125	0.332356770833333\\
0.125	0.333333333333333\\
0.125	0.334309895833333\\
0.125	0.335286458333333\\
0.125	0.337239583333333\\
0.125	0.341145833333333\\
0.125	0.348958333333333\\
0.125	0.364583333333333\\
0.125	0.395833333333333\\
0.125	0.458333333333333\\
0.1259765625	0.317708333333333\\
0.1259765625	0.325520833333333\\
0.1259765625	0.329427083333333\\
0.1259765625	0.331380208333333\\
0.1259765625	0.332356770833333\\
0.1259765625	0.333333333333333\\
0.1259765625	0.334309895833333\\
0.1259765625	0.335286458333333\\
0.1259765625	0.337239583333333\\
0.1259765625	0.341145833333333\\
0.1259765625	0.348958333333333\\
0.1259765625	0.364583333333333\\
0.1259765625	0.395833333333333\\
0.1259765625	0.458333333333333\\
0.126953125	0.317708333333333\\
0.126953125	0.325520833333333\\
0.126953125	0.329427083333333\\
0.126953125	0.331380208333333\\
0.126953125	0.332356770833333\\
0.126953125	0.333333333333333\\
0.126953125	0.334309895833333\\
0.126953125	0.335286458333333\\
0.126953125	0.337239583333333\\
0.126953125	0.341145833333333\\
0.126953125	0.348958333333333\\
0.126953125	0.364583333333333\\
0.126953125	0.395833333333333\\
0.126953125	0.458333333333333\\
0.12890625	0.317708333333333\\
0.12890625	0.325520833333333\\
0.12890625	0.329427083333333\\
0.12890625	0.331380208333333\\
0.12890625	0.332356770833333\\
0.12890625	0.333333333333333\\
0.12890625	0.334309895833333\\
0.12890625	0.335286458333333\\
0.12890625	0.337239583333333\\
0.12890625	0.341145833333333\\
0.12890625	0.348958333333333\\
0.12890625	0.364583333333333\\
0.12890625	0.395833333333333\\
0.12890625	0.458333333333333\\
0.1328125	0.317708333333333\\
0.1328125	0.325520833333333\\
0.1328125	0.329427083333333\\
0.1328125	0.331380208333333\\
0.1328125	0.332356770833333\\
0.1328125	0.333333333333333\\
0.1328125	0.334309895833333\\
0.1328125	0.335286458333333\\
0.1328125	0.337239583333333\\
0.1328125	0.341145833333333\\
0.1328125	0.348958333333333\\
0.1328125	0.364583333333333\\
0.1328125	0.395833333333333\\
0.1328125	0.458333333333333\\
0.140625	0.317708333333333\\
0.140625	0.325520833333333\\
0.140625	0.329427083333333\\
0.140625	0.331380208333333\\
0.140625	0.332356770833333\\
0.140625	0.333333333333333\\
0.140625	0.334309895833333\\
0.140625	0.335286458333333\\
0.140625	0.337239583333333\\
0.140625	0.341145833333333\\
0.140625	0.348958333333333\\
0.140625	0.364583333333333\\
0.140625	0.395833333333333\\
0.140625	0.458333333333333\\
0.15625	0.317708333333333\\
0.15625	0.325520833333333\\
0.15625	0.329427083333333\\
0.15625	0.331380208333333\\
0.15625	0.332356770833333\\
0.15625	0.333333333333333\\
0.15625	0.334309895833333\\
0.15625	0.335286458333333\\
0.15625	0.337239583333333\\
0.15625	0.341145833333333\\
0.15625	0.348958333333333\\
0.15625	0.364583333333333\\
0.15625	0.395833333333333\\
0.15625	0.458333333333333\\
0.1875	0.317708333333333\\
0.1875	0.325520833333333\\
0.1875	0.329427083333333\\
0.1875	0.331380208333333\\
0.1875	0.332356770833333\\
0.1875	0.333333333333333\\
0.1875	0.334309895833333\\
0.1875	0.335286458333333\\
0.1875	0.337239583333333\\
0.1875	0.341145833333333\\
0.1875	0.348958333333333\\
0.1875	0.364583333333333\\
0.1875	0.395833333333333\\
0.1875	0.458333333333333\\
0.25	0.317708333333333\\
0.25	0.325520833333333\\
0.25	0.329427083333333\\
0.25	0.331380208333333\\
0.25	0.332356770833333\\
0.25	0.333333333333333\\
0.25	0.334309895833333\\
0.25	0.335286458333333\\
0.25	0.337239583333333\\
0.25	0.341145833333333\\
0.25	0.348958333333333\\
0.25	0.364583333333333\\
0.25	0.395833333333333\\
0.25	0.458333333333333\\
};
\addplot [color=mycolor2, only marks, mark size=2.5pt, mark=x, mark options={solid, mycolor2}, forget plot]
  table[row sep=crcr]{%
0.109375	0.317708333333333\\
0.109375	0.325520833333333\\
0.109375	0.329427083333333\\
0.109375	0.331380208333333\\
0.109375	0.332356770833333\\
0.109375	0.333333333333333\\
0.109375	0.334309895833333\\
0.109375	0.335286458333333\\
0.109375	0.337239583333333\\
0.109375	0.341145833333333\\
0.109375	0.348958333333333\\
0.109375	0.364583333333333\\
0.109375	0.395833333333333\\
0.109375	0.458333333333333\\
0.1171875	0.317708333333333\\
0.1171875	0.325520833333333\\
0.1171875	0.329427083333333\\
0.1171875	0.331380208333333\\
0.1171875	0.332356770833333\\
0.1171875	0.333333333333333\\
0.1171875	0.334309895833333\\
0.1171875	0.335286458333333\\
0.1171875	0.337239583333333\\
0.1171875	0.341145833333333\\
0.1171875	0.348958333333333\\
0.1171875	0.364583333333333\\
0.1171875	0.395833333333333\\
0.1171875	0.458333333333333\\
0.12109375	0.317708333333333\\
0.12109375	0.325520833333333\\
0.12109375	0.329427083333333\\
0.12109375	0.331380208333333\\
0.12109375	0.332356770833333\\
0.12109375	0.333333333333333\\
0.12109375	0.334309895833333\\
0.12109375	0.335286458333333\\
0.12109375	0.337239583333333\\
0.12109375	0.341145833333333\\
0.12109375	0.348958333333333\\
0.12109375	0.364583333333333\\
0.12109375	0.395833333333333\\
0.12109375	0.458333333333333\\
0.123046875	0.317708333333333\\
0.123046875	0.325520833333333\\
0.123046875	0.329427083333333\\
0.123046875	0.331380208333333\\
0.123046875	0.332356770833333\\
0.123046875	0.333333333333333\\
0.123046875	0.334309895833333\\
0.123046875	0.335286458333333\\
0.123046875	0.337239583333333\\
0.123046875	0.341145833333333\\
0.123046875	0.348958333333333\\
0.123046875	0.364583333333333\\
0.123046875	0.395833333333333\\
0.123046875	0.458333333333333\\
0.1240234375	0.317708333333333\\
0.1240234375	0.325520833333333\\
0.1240234375	0.329427083333333\\
0.1240234375	0.331380208333333\\
0.1240234375	0.332356770833333\\
0.1240234375	0.333333333333333\\
0.1240234375	0.334309895833333\\
0.1240234375	0.335286458333333\\
0.1240234375	0.337239583333333\\
0.1240234375	0.341145833333333\\
0.1240234375	0.348958333333333\\
0.1240234375	0.364583333333333\\
0.1240234375	0.395833333333333\\
0.1240234375	0.458333333333333\\
};

\addplot [color=black, dotted, line width=1.0pt, forget plot]
  table[row sep=crcr]{%
0.109375	0.333333333333333\\
0.125	0.333333333333333\\
};

\addplot [color=black, line width=1.0pt, forget plot]
table[row sep=crcr]{%
	0.125	0.333333333333333\\
	0.25	0.333333333333333\\
};

\addplot [color=black, dotted, line width=1.0pt, forget plot]
  table[row sep=crcr]{%
0.125	0.317708333333333\\
0.125	0.333333333333333\\
};

\addplot [color=black, line width=1.0pt, forget plot]
table[row sep=crcr]{%
	0.125	0.333333333333333\\
	0.125	0.458333333333333\\
};

\end{axis}
\end{tikzpicture}%
		\caption{Comparison of $\Phi_{\text{eH}}$ and $\Phi_{\text{eH}, \delta}^{\pc}$ for varying parameters $k, \ell$.
		Markers {\color{blue}\raisebox{-0.3mm}{\scalebox{1.3}{$\circ$}}}/{\color{red} $\times$} indicate $|\Phi_{\text{eH}} - \Phi_{\text{eH}, \delta}^{\pc}|_\infty$ smaller/larger than $10^{-5}$.
		Black lines indicate analytically known parameter bounds.}
		\label{fig:HenckyParameterStudy}
	\end{center}
\end{figure}
These results confirm the polyconvexity of $W_{\text{eH}}$ if $k \geq \frac{1}{3}$ and $\ell \geq \frac{1}{8}$. While, according to our numerical investigation, the bound for $\ell$ seems sharp in the sense that polyconvexity is lost for $\ell <\frac{1}{8}$, values of $k<\frac{1}{3}$ seem to allow polyconvexity.

\section{Conclusion}\label{sec:conclusion}
We have shown that the computational efficiency of algorithms for polyconvex isotropic energy densities can be significantly improved. Based on the characterization of isotropic polyconvexity in terms of the signed singular value instead of the full matrix input \cite{WiePet23}, we have shown how the corresponding dimensional reduction from $d^2$- to $d$-dimensional space can be realized algorithmically. 
The convexification of the lifted space can be performed by computational geometry algorithms or linear programming. 
Both variants not only have minimal complexity in representative benchmarks but also allow numerical investigation of exponentiated Hencky-logarithmic energy densities and their polyconvexity properties for a range of parameters beyond those known analytically.

\section*{Acknowledgement} Fruitful discussions with Daniel Balzani and Maximilian Köhler are greatly acknowledged.

\bibliographystyle{alpha}
\bibliography{singularValuePolyconvexification}

\end{document}